\documentclass[leqno,11pt,a4paper]{amsart}
\usepackage[full]{textcomp}
\usepackage{newpxtext}
\usepackage{cabin} % sans serif
\usepackage[varqu,varl]{inconsolata} % sans serif typewriter
\usepackage[bigdelims,vvarbb]{newpxmath} % bb from STIX
\usepackage[cal=cm,bb=esstix,bbscaled=1.05]{mathalfa} % mathcal
\usepackage[margin=2.4cm]{geometry}
\usepackage{savesym}

\usepackage{mathabx}
\usepackage[usenames,dvipsnames]{color}
\usepackage{hyperref}
\usepackage{tikz}
\usetikzlibrary{cd}
\usepackage{graphicx}
\usepackage[all,cmtip]{xy}
\usepackage{mleftright}
\usepackage{booktabs}
\usepackage{cite}
\usepackage{enumitem}
\usepackage{mathdots}
\usepackage{eucal}
\setlist[enumerate]{labelsep=*, leftmargin=1.5pc}
\setlist[enumerate]{label=\normalfont(\roman*), ref=\roman*}
\newtheorem{thm}{Theorem}[section]
\newtheorem{lemma}[thm]{Lemma}
\newtheorem{cor}[thm]{Corollary}
\newtheorem{prop}[thm]{Proposition}

\theoremstyle{definition}
\newtheorem{example}[thm]{Example}
\newtheorem{remark}[thm]{Remark}
\newtheorem{definition}[thm]{Definition}
\newtheorem{notation}[thm]{Notation}
\newtheorem{question}[thm]{Question}
\newtheorem{problem}[thm]{Problem}

\newtheorem{construction}[thm]{Construction} 
\numberwithin{equation}{section}
\newcommand{\Z}{\mathbb{Z}}

\newcommand{\R}{\mathbb{R}}
\newcommand{\C}{\mathbb{C}}

\newcommand{\vol}[1]{\operatorname{vol}\mleft({#1}\mright)}

\newcommand{\on}{\operatorname}

\newcommand{\mat}{\left(\begin{array}}
\newcommand{\tam}{\end{array}\right)}

\newcommand{\twSp}{\widetilde{\on{Sp}}}

\newcommand{\Ru}{\on{Ru}}

\newcommand{\CZ}{\on{CZ}}
\newcommand{\LCZ}{\on{LCZ}}

\newcommand{\GL}{\on{GL}}
\newcommand{\Sp}{\on{Sp}}
\newcommand{\U}{\on{U}}

\newcommand{\On}{\on{O}}

\newcommand{\Spa}{\mathfrak{sp}}

\definecolor{dark green}{rgb}{0.0, 0.6, 0.0}

\graphicspath{{images/}}
\begin{document}
%-------------------------------------------------------------------------------
\author[J.~Chaidez]{J.~Chaidez}
\address{Department of Mathematics\\Princeton University\\Princeton, NJ\\08544\\USA}
\email{jchaidez@princeton.edu}

\author[O.~Edtmair]{O.~Edtmair}
\address{Department of Mathematics\\University of California at Berkeley\\Berkeley, CA\\94720\\USA}
\email{oliver\_edtmair@berkeley.edu}
%-------------------------------------------------------------------------------
\title[The Ruelle Invariant And Convexity In Higher Dimensions]{The Ruelle Invariant And Convexity In Higher Dimensions}

\begin{abstract} We construct the Ruelle invariant of a volume preserving flow and a symplectic cocycle in any dimension and prove several properties. In the special case of the linearized Reeb flow on the boundary of a convex domain $X$ in $\R^{2n}$, we prove that the Ruelle invariant $\Ru(X)$, the period of the systole $c(X)$ and the volume $\vol{X}$ satisfy
\[\Ru(X) \cdot c(X) \le C(n) \cdot \vol{X}\]
Here $C(n) > 0$ is an explicit constant dependent on $n$. As an application, we construct dynamically convex contact forms on $S^{2n-1}$ that are not convex, disproving the equivalence of convexity and dynamical convexity in every dimension.
\end{abstract}

\maketitle
%-------------------------------------------------------------------------------

\section{Introduction} \label{sec:introduction} In \cite{r1985}, Ruelle introduced his eponymous \emph{Ruelle invariant} $\on{Ru}(Y,\phi)$ of a flow $\phi: \R \times Y \to Y$ on a $3$-manifold $Y$ preserving a smooth measure $\mu$. This invariant is the integral of a function $\on{ru}(\phi)$ that (morally speaking) measures the linking of nearby trajectories of $\phi$ in $Y$.

\begin{figure}[h]
\centering
\includegraphics[width=.8\textwidth]{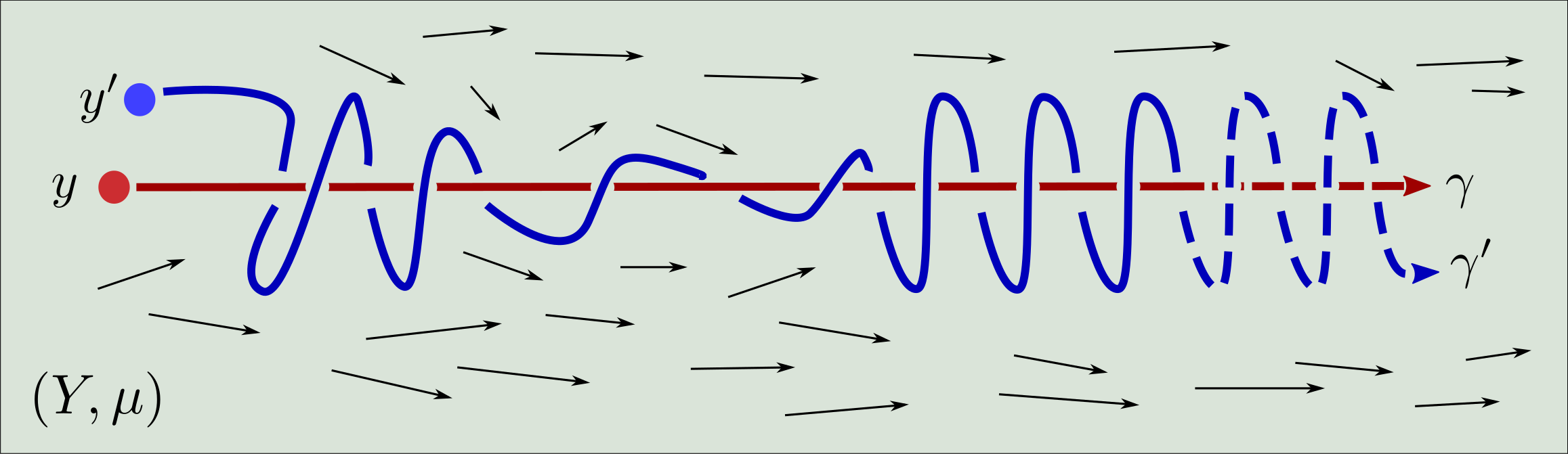}
\caption{The function $\on{ru}(\phi)$ at $y$ measures the time-averaged linking of the length $T$ trajectory $\gamma$ starting at $y$ and a nearby trajectory $\gamma'$, as $T$ goes to $\infty$.}
\label{fig:ru_picture}
\end{figure}
\noindent Since its introduction, the Ruelle invariant has appeared in low-dimensional dynamics (cf. Gambaudo-Ghys \cite{gg1997,gg2004,g2007}), bifurcation theory (cf. Parlitz \cite{p1993}) and Sturm-Liouville theory (cf. Schulz-Baldes \cite{s2007,s2012}). More recently, the Ruelle invariant has been applied very fruitfully to the study of 3-dimensional Reeb dynamics and 4-dimensional symplectic geometry \cite{h2019,ce2021,dgz2021}. 

\vspace{3pt}

In our previous work \cite{ce2021}, we applied the Ruelle invariant to find the first examples of contact forms on the $3$-sphere that are dynamically convex in the sense of Hofer-Wysocki-Zehnder \cite{hwz1998} but not symplectically convex (see Definition \ref{def:symplectic_convexity}). This was a longstanding unsolved problem, and remains particularly impervious to more conventional modern methods in symplectic geometry such as Floer theory.

\vspace{3pt}

In this paper, we initiate the study of the Ruelle invariant in higher dimensional Reeb dynamics. Specifically, we construct a substantial generalization of the Ruelle invariant in \cite{r1985} to symplectic cocycles of flows in any dimension. This generalization is related to previous ones such as the asymptotic Maslov index \cite{cgip2003}. We then formulate and prove higher dimensional versions of results in \cite{ce2021}, \cite{h2019} and \cite{dgz2021}. In particular, we show that dynamical convexity and symplectic convexity are inequivalent in all dimensions by constructing toric counter-examples, generalizing a construction of Dardennes-Gutt-Zhang \cite{dgz2021} from dimension four.

\subsection{Ruelle Invariant Of A Symplectic Cocycle} \label{subsec:ruelle_invariant_of_a_cocycle} Let us begin by summarizing our construction of the Ruelle invariant and discussing its important formal properties.

\vspace{3pt}

Let $Y$ be a compact manifold equipped with an autonomous flow $\phi:\R \times Y \to Y$ and let $E \to Y$ be a symplectic vector bundle. Also let $\rho:\R \times Y \to Y$ denote the obvious projection.

\begin{definition} A \emph{symplectic cocycle} $\Phi$ on $E$ for the flow $(Y,\phi)$ is a symplectic bundle map
\begin{equation}\Phi:\rho^*E \to \phi^*E \qquad\text{satisfying}\qquad \Phi(s + t,x) = \Phi(t,\phi(s,x)) \Phi(s,x)\end{equation}
\end{definition}

\noindent Fix a $\phi$-invariant Borel measure on $Y$, a symplectic cocycle $(E,\Phi)$ with vanishing first Chern class $c_1(E)$ and a homotopy class of trivialization $\tau:\Lambda E \simeq \C$ of the determinant line bundle $\Lambda E$. Here we consider the complex determinant line bundle with respect to an auxiliary choice of compatible complex structure on $E$ (see \S \ref{subsec:rotation_function}). 

\begin{thm} \label{thm:Ruelle_invariant_intro} There is a well-defined Ruelle density and Ruelle invariant, denoted respectively by
\[\on{ru}(\Phi,\tau) \in L^1(Y,\mu) \qquad \text{and}\qquad \on{Ru}(\Phi,\tau,\mu) := \int_Y \on{ru}(\Phi,\tau) \cdot \mu \]
Moreover, the Ruelle density and invariant satisfy the following properties.
\begin{itemize}
    \item[(a)] (Covariance) If $\Psi:(E,\Phi) \to (E',\Phi')$ is a symplectic cocycle isomorphism that maps $\tau$ to $\tau'$, then
    \[\on{ru}(\Phi,\tau) = \on{ru}(\Phi',\tau') \qquad\qquad \on{Ru}(\Phi,\tau,\mu) = \on{Ru}(\Phi',\tau',\mu)\]
    \item[(b)] (Direct Sum) If $\Phi = \Phi_1 \oplus \Phi_2$ is a direct sum of symplectic cocycles and $\tau = \tau_1 \otimes \tau_2$, then
    \[\on{ru}(\Phi_1 \oplus \Phi_2,\tau_1 \otimes \tau_2) = \on{ru}(\Phi_1,\tau_1) + \on{ru}(\Phi_2,\tau_2)\]
    \item[(c)] (Linearity) If $a\mu + b\nu$ is a positive combination of $\phi$-invariant Borel measures $\mu$ and $\nu$, then
    \[
    \on{Ru}(\Phi,\tau,a\mu + b\nu) = a\on{Ru}(\Phi,\tau,\mu) + b\on{Ru}(\Phi,\tau,\nu)\]
    \item[(d)] (Trivial Bundle) If $\Phi$ is a symplectic cocycle on $\C^n$ with the tautological trivialization $\tau_{\on{std}}$, then
    \[\on{ru}(\Phi,\tau_{\on{std}}) = \lim_{T \to \infty} \frac{q \circ \widetilde{\Phi}_T}{T} \qquad\qquad \on{Ru}(\Phi,\tau_{\on{std}}) = \lim_{T \to \infty} \frac{1}{T} \int_Y q \circ \widetilde{\Phi}_T \cdot \mu\]
    Here $q$ is any rotation quasimorphism (see \S \ref{subsec:rotation_quasimorphisms}) and $\widetilde{\Phi}:\R\times Y \to \widetilde{\on{Sp}}(2n)$ is the lift of $\Phi$ (regarded as a map $\R \times Y \to \on{Sp}(2n)$) to the universal cover $\widetilde{\on{Sp}}(2n)$.
\end{itemize}
\end{thm}

The data needed to apply Theorem \ref{thm:Ruelle_invariant_intro} arises in a fairly natural way for the dynamical systems that arise in symplectic geometry. Here are the main examples of interest.

\begin{example}[Symplectic Flows] \label{ex:Ru_of_symplectic_flow} Let $(X,\omega)$ be a symplectic manifold with a compact symplectic submanifold $\Sigma \subset X$ and let $V$ be a complete symplectic vector field tangent to $\Sigma$.

\vspace{3pt}

The differential of the symplectic flow $\Phi$ generated by $V$ induces a symplectic cocycle
\[T\Phi:\R \times TX \to \Phi^*TX\] 
The flow $\Phi$ preserves $\Sigma$ since $V$ is tangent to $\Sigma$. Moreover, $\Sigma$ is equipped with the natural invariant measure $\omega^m|_\Sigma$ where $\on{dim}(\Sigma) = 2m$. Given a homotopy class of trivialization $\tau:\Lambda(TX) \simeq \C$ along $\Sigma$, we thus acquire a Ruelle density and invariant via Theorem \ref{thm:Ruelle_invariant_intro}.
\[\on{ru}(T\Phi|_\Sigma,\tau)  
\qquad \text{and}\qquad \on{Ru}(T\Phi|_\Sigma,\tau,\omega^m|_\Sigma)\]
More generally, we only need to assume that the flow $\Phi$ is defined near $\Sigma$ and that $\Sigma$ is a (not necessarily symplectic) submanifold equipped with an invariant measure $\mu$. This special case is discussed in \cite{cgip2003}. 
\end{example}

\begin{example}[Hamiltonian Flows] Let $X$ be a compact symplectic manifold with boundary and let $H:X \to \R$ be a Hamiltonian that is locally constant on $\partial X$. Assume that $c_1(X) = 0$.

\vspace{3pt}

Then as a special case of Example \ref{ex:Ru_of_symplectic_flow}, we get a Ruelle invariant associated to $X$, the flow $\Phi^H$ of $H$ and a chosen homotopy class of trivialization $\tau$. We denote this by
\[\on{Ru}(X,H,\tau) \qquad\text{or simply }\quad\on{Ru}(X,H)\quad\text{if}\quad H^1(X;\Z) = 0\]
\end{example}

\begin{example}[Reeb Flows] Recall that a contact $(2n-1)$-manifold $(Y,\xi)$ is a manifold equipped with a $(2n-2)$-plane field $\xi \subset TY$, called the contact structure, that is the kernel of a contact form $\alpha$. A contact form on $Y$ is a $1$-form that satisfies
\[\ker(d\alpha) \subset TY \text{ is rank 1} \qquad\text{and}\qquad \alpha|_{\ker(d\alpha)} > 0\]
Every contact form comes equipped with a natural Reeb vector field $R$, defined by
\[\alpha(R) = 1 \qquad \iota_Rd\alpha = 0\]
The flow $\Phi:\R \times Y \to Y$ of the Reeb vector field is simply called the Reeb flow of $Y$. Note that $\Phi$ preserves $\alpha$ and the natural volume form $\alpha \wedge d\alpha^{n-1}$. The contact structure $\xi$ of $Y$ is a symplectic vector bundle with symplectic form $d\alpha|_\xi$. Thus
\[(\xi,T\Phi|_\xi)\]
has the structure of a symplectic cocycle. If $\xi$ has vanishing first Chern class, we can choose a homotopy class of trivialization $\tau:\Lambda \xi \simeq \C$ to acquire a Ruelle invariant, denoted in this case by
\[
\on{Ru}(Y,\alpha,\tau) \qquad\text{or simply}\qquad \on{Ru}(Y,\alpha) \quad\text{if}\quad H^1(Y;\Z) = 0
\]
\end{example}

\subsection{Ruelle Invariant Of Liouville Domains} \label{subsec:convex_Reeb_flows} In the case of Liouville domains, the Ruelle invariant yields a new symplectomorphism invariant (under some mild topological hypotheses).

\vspace{3pt}

Recall that a \emph{Liouville domain} $(X,\lambda)$ is a compact symplectic manifold $(X,\omega)$ with a vector field $Z$ and a symplectically dual $1$-form $\lambda = \iota_Z\omega$ such that
\[\omega = d\lambda \qquad\text{and}\qquad Z \text{ points outward along }\partial X\]
The 1-form $\lambda$ and the vector field $Z$ are called the \emph{Liouville form} and \emph{Liouville vector field} of $X$. The \emph{skeleton} $\on{Skel}(X)$ of a Liouville domain $(X,\lambda)$ is the set given by
\[
\on{Skel}(X) = \bigcap_{t < 0} \Phi^Z_t(X) \quad\text{where}\qquad \Phi^Z \text{ is the flow generated by }Z
\] 
The boundary $\partial X$ of a Liouville domain $X$ is a contact manifold with contact form $\lambda|_{\partial X}$. Moreover, $X$ admits a canonical Hamiltonian on the complement of the skeleton 
\[
H_X:X \setminus \on{Skel}(X) \to (0,1] \qquad \text{characterized by}\qquad ZH_X = H_X \text{ and }H_X^{-1}(1) = \partial X
\]
The level sets of $H_X$ are canonically contactomorphic to $\partial X$ and the Hamiltonian vector field of $H_X$ agrees with the Reeb vector field of $\lambda|_{\partial X}$ on each level. Note that $H_X$ extends continuously to the skeleton as $H_X|_{\on{Skel}(X)} = 0$, but in general this extension is not differentiable. 

\begin{figure}[h]
\centering
\includegraphics[width=.9\textwidth]{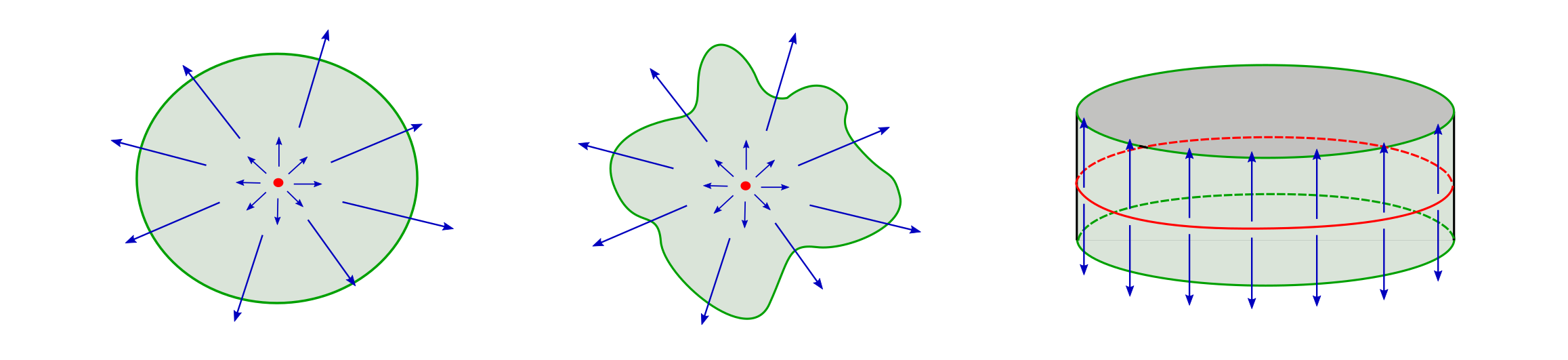}
\caption{Some pictures of 2-dimensional Liouville domains. The skeleton is depicted in red and the Liouville vector field in blue.}
\label{fig:liouville_domain}
\end{figure}

\begin{example}[Star-Shaped Domains] A \emph{star-shaped} domain $X \subset \C^n$ with smooth boundary $Y$ is a domain such that
\[0 \in \on{int}(X) \qquad\text{and}\qquad \text{the radial vector field }\partial_r \text{ is transverse to }Y\]
A star-shaped domain is naturally a Liouville domain, via restriction of the standard symplectic form and Liouville vector field on $\C^n$. In standard coordinates, these are given by
\[\omega = \sum_{j=1}^n dx_j \wedge dy_j \qquad\text{and}\qquad Z = \frac{1}{2} \sum_j x_j \partial_{x_j} + y_j \partial_{y_j}  \qquad\text{where}\qquad z_j = x_j + i \cdot y_j \]
As with any Liouville domain, the restriction $\lambda|_Y$ of the Liouville form $\lambda = \iota_Z\omega$ is a contact form on the boundary $Y$. The Reeb vector field $R$ on $Y$ is given 
\[R = \frac{J\nu}{\langle Z,\nu\rangle} \qquad\text{where }\nu\text{ is the normal vector field and }J\text{ is multiplication by }i\] \end{example}

Strictly speaking, the Ruelle invariant of $H_X$ is not well-defined since $H_X$ is only defined away from $\on{Skel}(X)$. However we can show that $\on{ru}(X,H_X)$ is invariant under $Z$. Thus we can take
\[
\on{Ru}(X,H_X) := \lim_{E \to 0} \; \on{Ru}(X_E,H_X) \qquad\text{where}\qquad X_E = H_X^{-1}[E,1]
\]By applying a standard argument using Grey stability (cf. \cite[Lemma 3.5]{ce2021} or \cite[\S 3]{dgz2021}) along with Stokes theorem, one may prove the following result.

\begin{lemma}\label{lem:ruelle_liouville_domain} Let $(X,\lambda)$ be a Liouville domain with $H^1(X;\Z) = H^2(X;\Z) = 0$. Then
\[\on{Ru}(X,H_X) = \on{Ru}(\partial X,\lambda|_{\partial X})\]
Furthermore, $\on{Ru}(X,H_X) = \on{Ru}(W,H_W)$ if $X$ and $W$ are symplectomorphic. \end{lemma}

\noindent Thus the Ruelle invariant $\on{Ru}(X) = \on{Ru}(X,H_X)$ is a symplectic invariant of Liouville domains. 

\begin{example}[Toric Domains] In the case of a toric domain, we can prove an explicit formula for the Ruelle invariant that generalizes the formulas appearing in \cite{h2019,dgz2021}. 

\vspace{3pt}

Let $X_\Omega \subset \C^n$ be a smooth, star-shaped toric domain with moment region $\Omega \subset [0,\infty)^n$. Let $f_\Omega:[0,\infty)^n \to [0,\infty)$ be the unique smooth function such that
\[f^{-1}_\Omega[0,1] = \Omega \qquad\text{and}\qquad \sum_i x_i \cdot \partial_i f_\Omega = f_\Omega\]

\begin{prop} (Proposition \ref{prop:toric_ruelle}) The Ruelle invariant of $X_\Omega$ is given by the following formula.
\[
\on{Ru}(X_\Omega) = \sum_i \int_\Omega \partial_i f_\Omega \cdot \on{dvol}_{\R^n}
\]
\end{prop}

\noindent We will provide a review of toric domains and their Reeb dynamics in \S \ref{sec:toric_formula}.
\end{example}

\subsection{Symplectic Convexity} \label{subsec:convexity} Our main application of the Ruelle invariant is to distinguish symplectically convex domains from dynamically convex domains. Let us recall the former concept. 

\begin{definition}\label{def:symplectic_convexity} A star-shaped domain $X$ is \emph{symplectically convex} if it is symplectomorphic to a convex star-shaped domain $X'$.
\end{definition}

\noindent Convex domains and their contact boundaries have many special properties that distinguish them from ordinary star-shaped domains and arbitrary contact forms on the sphere, particularly in dimension four (cf. \cite{hwz1998,v2000,ghr2020}).

\vspace{3pt}

In \cite{ce2021}, we demonstrated a new special property of the Ruelle invariant of convex star-shaped domains. To be precise, let $c(X)$ denote the period of the systole of $\partial X$, i.e.
\begin{equation}
c(X) := \on{min}\{T \; : \; T \text{ is the period of a closed Reeb orbit on }\partial X\}
\end{equation}

\begin{thm} \cite{ce2021} \label{thm:4d_ruelle_bound} There are constants $C,c > 0$ such that, for any convex star-shaped domain $X \subset \C^2$
\[c \cdot \on{vol}(X) \le \Ru(X) \cdot c(X) \le C \cdot \on{vol}(X)\]
\end{thm}

\noindent Our second main result in this paper is the generalization of the upper bound in Theorem \ref{thm:4d_ruelle_bound}.

\begin{thm} \label{thm:main_inequality} There is a constant $C(n) > 0$ such that any convex star-shaped domain $X \subset \C^n$ satisfies
\[\Ru(X) \cdot c(X) \le C(n) \cdot \on{vol}(X)\]
\end{thm}
\noindent Let us briefly sketch the proof, which is strategically similar to the proof of Theorem \ref{thm:4d_ruelle_bound} in \cite{ce2021}.

\vspace{3pt}

\emph{Proof Sketch.} We start by observing that the tangent cocycle $T\Phi$ induced by the Hamiltonian flow $\Phi$ of $H_X$ is generated by the Hessian of $H_X$, in the sense that
\[\frac{d}{dt}(T\Phi)(t,x) = J\circ A(\Phi(t,x)) \circ T\Phi(t,x) \qquad\text{where}\qquad A(x) = \nabla^2 H_X(t,x)\]
General properties of the rotation quasimorphism (see \S \ref{subsec:rotation_quasimorphisms}) imply a trace estimate for the Ruelle invariant (see Proposition \ref{prop:Ruelle_invariant}(e)) when the generator $A$ is positive semi-definite, which is the case if $X$ is convex. Thus we get
\[
\on{Ru}(X) \le \frac{8n^2}{\pi} \cdot S(H_X) \qquad\text{where}\qquad S(H_X) := \int_X \on{tr}(A) \cdot \omega^n = \int_X \Delta H_X \cdot \omega^n
\]
By analyzing the functional $S$, we prove (Proposition \ref{prop:sandwich_estimate}) that if $X$ and $W$ are sandwiched, in the sense that $W \subset X \subset c \cdot W$ for some constant $c \ge 1$, then
\[S(H_X) \le C(c,n) \cdot S(H_W) \qquad\text{where $C(c,n)$ depends only on $n$ and $c$}\]
On the other hand, by the John ellipsoid theorem, we can find a standard symplectic ellipsoid $E$ such that $E \subset X \subset 2n \cdot E$ (after applying a symplectomorphism to $X$). For this ellipsoid, we have
\[
c(X) \le c(2n \cdot E) \qquad \on{Ru}(X) \le C'(n) \cdot S(H_{2n \cdot E}) \qquad \on{vol}(2n \cdot E) \le (2n)^{2n} \cdot \on{vol}(X)
\]
This reduces the proof to the statement that $S(H_E) \cdot c(E) \le C''(n) \cdot \on{vol}(E)$ for any standard ellipsoid $E$ and a constant $C''(n)$ depending on $n$. This is a simple calculation (Lemma \ref{lem:ellipsoid_quantities}). \qed

\vspace{3pt}

\noindent We will carry our a detailed version of this proof (keeping track of constants) in \S \ref{subsec:proof_of_main_inequality}.

\begin{remark} Our proofs of Theorem \ref{thm:main_inequality} above and Theorem \ref{thm:4d_ruelle_bound} in \cite{ce2021} are very similar. 

\vspace{3pt}

The key difference is our use of the Laplacian integral in place of the total mean curvature of the contact boundary, which plays an almost identical role in \cite{ce2021}. A higher dimensional bound by some extrinsic curvature integral (cf. \cite[Lemma 3.11]{ce2021}) would, likely, further simplify and improve the proof of Theorem \ref{thm:main_inequality}. At this time, we do not have a construction of the Ruelle invariant in higher dimensions that makes such a bound manifest.

\vspace{3pt}

Relatedly, the lower bound in Theorem \ref{thm:4d_ruelle_bound} is likely true in higher dimensions. However, at this time, it is not clear how to adapt our methods to prove it. The proof in \cite{ce2021} relies on some delicate geometric reasoning specific to $\C^2$. 
\end{remark}

\begin{remark} \label{rmk:main_thm_constant} The constant $C(n)$ in Theorem \ref{thm:main_inequality} can be given explicitly as
\[C(n) := 2^{2n+5} \cdot n^{2n+3} \cdot \on{exp}(8n^4)\]
We believe that this constant is far from optimal. In fact, in dimension 4, it is an inferior constant to the one produced in the upper bound of Theorem \ref{thm:4d_ruelle_bound} in \cite{ce2021}. \end{remark}

\subsection{Dynamical Convexity} Symplectic convexity is a mysterious and fundamentally extrisic condition that nonetheless plays a fundamental role in the symplectic geometry of star-shaped domains. One is thus drawn naturally to the following problem.

\begin{problem} \label{problem:what_is_convexity} Give a characterization of symplectic convexity in terms of symplectomorphism invariant properties, i.e. without referencing an embedding to $\C^n$.\end{problem}

A prominent candidate criterion to resolve Problem \ref{problem:what_is_convexity} was introduced by Hofer-Wysocki-Zehnder in their groundbreaking paper \cite{hwz1998}. This characterization uses the lower-semicontinuous extension $\LCZ$ (see \S \ref{subsec:CZ_index}) of the Conley-Zehnder index $\CZ$, which can be viewed as a sort of Floer-theoretic Morse index of a closed Reeb orbit.

\begin{definition} A contact form on $\alpha$ on $S^{2n-1}$ is \emph{dynamically convex} if
\[\LCZ(\gamma) \ge n+1 \qquad\text{for every closed Reeb orbit }\gamma\text{ of }\alpha\]
Likewise, a star-shaped domain $(X,\lambda)$ is dynamically convex if $(\partial X,\lambda|_{\partial X})$ is. \end{definition}

\noindent Since \cite{hwz1998}, dynamical convexity has been used as a key hypothesis for many results in symplectic geometry (cf. \cite{z2019a,z2019b,am2014,am2015,gg2016,fvk2018,hn2014}). It is simple to check that every strictly positively curved convex domain is dynamically convex, but the converse has been open for more than 20 years.

\begin{question} \label{qu:dynamical_convex_is_convex} Is every dynamically convex contact form on $S^{2n-1}$ also convex?
\end{question} 

\noindent In dimension four, we resolved this problem in \cite{ce2021} by constructing examples of dynamically convex contact manifolds violating both bounds in Theorem \ref{thm:4d_ruelle_bound}.

\vspace{3pt}

There was substantial evidence prior to \cite{ce2021} that the answer to Question \ref{qu:dynamical_convex_is_convex} is no. For example, Abbondandolo-Bramham-Hryniewicz-Salom\~{a}o proved in \cite{abhs2} that the weak Viterbo conjecture fails for dynamically convex domains.  There is substantial evidence for the latter conjecture, especially in dimension four \cite{ch2020}, so the contact forms in \cite{abhs2} are likely not convex. 

\vspace{3pt}

In higher dimensions, Ginzburg-Macarini \cite{gm2020} constructed examples of dynamically convex contact forms admitting an action of a finite group $G$ that were not $G$-equivariantly isomorphic to a convex boundary with a similar $G$-action. However, their methods only apply when $G$ is non-trivial, and thus do not answer Question \ref{qu:dynamical_convex_is_convex}.

\vspace{3pt}

Theorem \ref{thm:main_inequality} can be used to resolve Question \ref{qu:dynamical_convex_is_convex} in any dimension. In fact, using the results in \cite{ce2021}, Dardennes-Gutt-Zhang \cite{dgz2021} introduced an elegant toric construction of non-convex, dynamically convex domains in $\C^2$ that is much simpler than the open book construction in \cite{ce2021}. Using a straight forward adaptation of their operation, we prove the following result.

\begin{prop} (Proposition \ref{prop:toric_counter_examples}) \label{prop:toric_counter_examples_intro} Let $X_\Omega$ be a star-shaped, concave toric domain. Then for any $C,\epsilon > 0$, there is a smooth, star-shaped, concave moment region
\[\hat{\Omega} \supset \Omega\]
that satisfies the following properties
\[\on{vol}(X_\Omega) \le \on{vol}(X_{\hat{\Omega}}) \le \on{vol}(X_\Omega) + \epsilon \qquad \on{Ru}(X_{\hat{\Omega}}) \ge C \qquad c(X_\Omega) \le c(X_{\hat{\Omega}})\]
\end{prop}

\noindent Smooth concave toric domains are examples of strictly monotone toric domains (see Definition \ref{def:strictly_monotone}), which are all dynamically convex (Proposition \ref{prop:strictly_monotone_implies_dynamically_convex}). Therefore, Proposition \ref{prop:toric_counter_examples} resolves Question \ref{qu:dynamical_convex_is_convex} as it implies the following corollary.

\begin{cor} There are dynamically convex contact forms on $S^{2n-1}$ that are not symplectically convex. 
\end{cor}

\subsection*{Outline} This concludes the introduction {\bf \S \ref{sec:introduction}}. The rest of the paper is organized as follows.

\vspace{3pt}

In {\bf \S \ref{sec:symplectic_linear_algebra}}, we discuss preliminaries from symplectic linear algebra: the polar decomposition (\S \ref{subsec:polar_decomposition}), the rotation quasimorphism (\S \ref{subsec:rotation_quasimorphisms}) and Conley-Zehnder indices (\S \ref{subsec:CZ_index}-\ref{subsec:indices_of_orbits}).

\vspace{3pt}

In {\bf \S \ref{sec:Ruelle_density}}, we carry out the construction of the Ruelle invariant in detail. We start by discussing the construction of the rotation function via sub-additive ergodic theory (\S \ref{subsec:rotation_function}). Then we construct the Ruelle invariant and demonstrate its properties (\S \ref{subsec:construction_of_ruelle}). 

\vspace{3pt}

In {\bf \S \ref{sec:Ruelle_invariant_of_convex_domains}}, we prove the main estimate on the Ruelle invariant. The begin with some preliminary estimates from Riemannian geometry (\S \ref{subsec:linear_tensor_fields}-\ref{subsec:laplacian_functional}). We then review some properties of ellipsoids (\S \ref{subsec:ellipsoids}) before proceeding to the main proof (\S \ref{subsec:proof_of_main_inequality}).

\vspace{3pt}

In {\bf \S \ref{sec:toric_formula}}, we construct our toric counter-examples. First, we introduce toric domains and prove a Ruelle invariant formula (\S \ref{subsec:toric_domains}). We then discuss monotone (\S \ref{subsec:monotone_domains}) and concave (\S \ref{subsec:concave_domains}) toric domains. Finally, we construct our counter-example in the last section of the paper (\S \ref{subsec:counterexamples}). 

\subsection*{Acknowledgements} We would like to thank Lior Alon for helpful conversation. JC was supported by the National Science Foundation under Award No. 2103165. 

\section{Symplectic Linear Algebra} \label{sec:symplectic_linear_algebra} In this section, we review background topics from symplectic linear algebra that will be required later in the paper.

\vspace{3pt}

 Specifically, we discuss polar decompositions and rotation quasimorphisms, which are key ingredients in our construction of the Ruelle invariant in \S \ref{sec:Ruelle_density}. We also discuss variants and properties of the Conley-Zehnder index, which will be needed in \S \ref{sec:toric_formula}.

\subsection{Polar Decomposition} \label{subsec:polar_decomposition} Recall that every matrix $A\in\on{GL}(n,\R)$ admits a unique \emph{polar decomposition} into a product $A=UP$ where $U$ is orthogonal and $P$ is symmetric positive definite. 

\vspace{3pt}

We can view the polar decomposition as pair of smooth maps between spaces of matrices.
\begin{equation}
\Phi = (U,P):\on{GL}(n) \to \on{O}(n) \times \on{P}(n) 
\end{equation}
Here $\on{O}(n)$ and $\on{P}(n)$ are, respectively, the spaces of orthogonal matrices and symmetric matrices.
\[
\on{O}(n) = \{A \in \on{GL}(n) \; : \; AA^T = \on{Id}\} \qquad\text{with Lie algebra}\qquad \mathfrak{o}(n) = \{A \in \mathfrak{gl}(2n) \; : \; A + A^T = 0 \}
\]
\[
\on{P}(n) := \{A \in \on{GL}(n) \; : \; A = A^T\} \qquad \text{with tangent space}\qquad \mathfrak{p}(n) := \{A \in \mathfrak{gl}(n) \; : \; A = A^T\}
\]
We will need an explicit integral expression for the derivative of the polar decomposition. 

\begin{lemma} \label{lem:polar_derivative} The differential $TU:T\on{GL}(n) \to T\On(n)$ of the map $U:\on{GL}(n) \to \On(n)$ is given by
\[TU_A(UB) = U \cdot \int_0^\infty e^{-sP}(B - B^T)e^{-sP} ds \quad\text{at}\quad A = UP\]
\end{lemma}

\begin{proof} Fix $A \in \GL(2n)$ and let $A = UP$ be the polar decomposition. Note that we can split the tangent space $T_A\GL(2n)$ into a direct sum
\[
T_A\GL(2n) = \R^{n\times n} = U \cdot \mathfrak{o}(n) + U \cdot \mathfrak{p}(n)
\]
That is, any matrix can be written as a sum $US + UT$ where $S$ is anti-symmetric and $T$ is symmetric. Clearly, if $T \in \mathfrak{p}(n)$ is a small symmetric matrix $T \in \mathfrak{p}(n)$, the unitary part of $U(P + T)$ is $U$. Thus, $U \cdot \mathfrak{p}(n)$ is the kernel of $TU_A$, and so
\[TU(UB) = TU_A(US) \qquad\text{where}\qquad S = \frac{1}{2}(B - B^T)\]

Thus, we must compute $TU_A(US)$ where $S$ is the anti-symmetric part of $B$. Let
\[T\Phi_A(US) = (UM,N) \qquad\text{with}\qquad M \in \mathfrak{o}(n) \text{ and }N \in \mathfrak{p}(n)\]
denote the image of $US$ under $T\Phi_A$. Essentially by definition, $M$ and $N$ are the unique matrices that satisfy
\[US = T(\Phi^{-1})_{U,P}(UM,N) = \frac{d}{dt}(Ue^{Mt}P + U (P+tN))= UMP + UN\]
Multiplying this equation by $U^{-1}$ and taking the transpose, we acquire the two equations
\[S = MP + N \qquad\text{and}\qquad -S = -PM + N\]
The difference of these two equations is the well-known Lyupanov equation for $M$.
\[2S = \{M,P\} = MP + PM\]
This equation has an integral solution (cf. \cite[Thm 12.3.3 and Thm 13.1.1]{l1985}) given by
\[M = 2\int_0^\infty e^{-Ps}Se^{-Ps} ds = \int_0^\infty e^{-Ps}(B - B^T)e^{-Ps} ds \]
By construction of $M$, we have $TU_A(UB) = UM$, so this is the desired formula. 
\end{proof}

We are, of course, mostly interested in the symplectic polar decomposition. Let $\Omega$ denote the standard linear symplectic structure on $\C^n \simeq \R^{2n}$, i.e.
\[\Omega = \left[\begin{array}{cc}
0 & -I_n\\
I_n & 0
\end{array}\right]\]
We abbreviate the group of linear symplectomorphisms on $(\C^n,\Omega)$ and its Lie algebra in the usual way. 
\[
\Sp(2n) = \{A \in \on{GL}(2n) \; : \; A\Omega A^T = \Omega\} \qquad\text{and}\qquad \Spa(2n) = \{A \in \mathfrak{gl}(2n) \; : \; A\Omega + \Omega A^T = 0 \}
\]
Recall that $\mathfrak{sp}(2n) = \Omega \cdot \mathfrak{p}(2n)$ or, in other words, that $A$ is in the symplectic Lie algebra if and only if $\Omega A$ is symmetric. We let
\[\on{U}(n) = \on{O}(2n) \cap \on{Sp}(2n)\]
denote the unitary group on $\C^n$. By standard linear algebra (cf. \cite[Ch. 2]{ms2017}), the polar decomposition restricts to a map
\[
(U,P):\Sp(2n) \to \on{U}(n) \times (P(2n) \cap \on{Sp}(2n))
\]

The derivative formula in Lemma \ref{lem:polar_derivative} implies an estimate for the trace of derivative of the polar decomposition. This will be a key ingredient for bounding the Ruelle invariant in \S \ref{subsec:construction_of_ruelle}.

\begin{lemma}[Trace Estimate] \label{lem:trace_estimate} Let $A$ be a symplectic matrix and let $\Omega S$ be a symplectic Lie algebra element with $S$ positive semi-definite. Then
\[|\on{tr}_\C(TU_A(\Omega SA) \cdot U^{-1})| \le 16n^2 \cdot \on{tr}_\R(S)\]
\end{lemma}

\begin{proof} First, note that we can compute the complex trace as a real trace, as follows.
\begin{equation} \label{eqn:trace_estimate_1}
\on{tr}_\C(TU_A(\Omega SA) \cdot U^{-1}) = 2i \cdot \on{tr}_\R(\Omega^{-1} \cdot TU_A(\Omega SA) \cdot U^{-1}) \end{equation}
Thus it suffices to estimate the real trace of $\Omega^{-1} \cdot TU_A(\Omega SA) \cdot U^{-1}$. We may write $\Omega SA = UB$ where $B = U^T\Omega SUP$ and apply Lemma \ref{lem:polar_derivative} to see that
\[
TU_A(\Omega SA) = U \cdot \int_0^\infty e^{-sP}(U^T\Omega SUP + P^TU^TS\Omega U)e^{-sP} ds
\]
We multiply on the left by $-\Omega $ and on the right by $U^{-1} = U^T$ to acquire the formula
\begin{equation} \label{eqn:trace_estimate_2}
\Omega^T \cdot TU_A(\Omega SA) \cdot U^T = \int_0^\infty (\Omega^TUe^{-sP}U^T\Omega)S(UPe^{-sP}U^T) + (\Omega^TUe^{-sP}PU^T\Omega)\Omega^T S\Omega(Ue^{-sP}U^T) ds
\end{equation}
The matrices $S,P$ and $e^{-sP}$ are all positive definite and $e^{-sP}$ has eigenvalues between $0$ and $1$. Thus, we may estimate the integrand on the righthand side as follows.
\[
|\on{tr}_\R((\Omega^TUe^{-sP}U^T\Omega)S(UPe^{-sP}U^T) + (\Omega^TUe^{-sP}PU^T\Omega)\Omega^TS\Omega(Ue^{-sP}U^T))|\]
\[\le |\on{tr}_\R((\Omega^TUe^{-sP}U^T\Omega)S(UPe^{-sP}U^T))| + |\on{tr}_\R((\Omega^TUe^{-sP}PU^T\Omega)\Omega^TS\Omega(Ue^{-sP}U^T))|\]
\[\le 2 \cdot \on{tr}_\R(e^{-sP}) \cdot \on{tr}_\R(S) \cdot \on{tr}(Pe^{-sP}) \le 4n \cdot \on{tr}_\R(S) \cdot \on{tr}(Pe^{-sP}) 
\]
Therefore, we have
\[
|\on{tr}_\R(\Omega^T \cdot TU_A(\Omega SA) \cdot U^T) \le 4n \cdot \on{tr}_\R(S) \cdot \big(\int_0^\infty Pe^{-sP} ds \big) = 4n \cdot \on{tr}_\R(S) \cdot \on{tr}_\R(I_{2n}) = 8n^2 \cdot \on{tr}_\R(S)
\]
We can plug this estimate into (\ref{eqn:trace_estimate_1}) to acquire the desired result. \end{proof}

\subsection{Rotation Quasimorphisms} \label{subsec:rotation_quasimorphisms} The rotation quasimorphism is a certain (equivalence class of) quasimorphism on the universal cover $\widetilde{\Sp}(2n)$ of $\Sp(2n)$. Let us recall the relevant definitions.

\vspace{3pt}

\begin{definition} A \emph{quasimorphism} $q:G \to \R$ from a group $G$ is a map that satisfies
\begin{equation} \label{eqn:quasimorphism_property}
|q(gh) - q(g) - q(h)| < C \qquad\text{for all }g,h \in G\text{ and some }C > 0\text{ independent of $g,h$}
\end{equation}
Two quasimorphisms $q$ and $q'$ are \emph{equivalent} if $|q - q'|$ is bounded, and $q$ is \emph{homogeneous} if
\[q(g^k) = k \cdot q(g) \qquad \text{for any}\qquad g \in G \qquad\text{and}\qquad k\in\Z\]
\end{definition}

The universal cover of the symplectic group possesses a canonical homogeneous quasimorphism, due to the following result of Salamon-Ben Simon \cite{ss2010}.

\begin{thm}[\cite{ss2010}, Thm 1] \label{thm:quasi_homo_twSp} There exists a unique homogeneous quasimorphism
\[\rho:\twSp(2n) \to \R\]
that restricts to the lift of the complex determinant $\on{det}_\C$ on $\on{U}(n)$. That is, the diagram
\begin{equation}
\begin{tikzcd}
\widetilde{\on{U}}(n) \arrow[r,"\rho"] \arrow[d,swap,"\pi"] & \R \arrow[d,"\on{exp}(2\pi i \cdot )"]\\
\on{U}(n) \arrow[r,"\on{det}_\C"] & U(1)
\end{tikzcd}\qquad \text{commutes.}
\end{equation}

\end{thm} 

\begin{definition} A \emph{rotation quasimorphism} $q:\twSp(2n) \to \R$ is a quasimorphism that is equivalent to the quasimorphism $\rho$ in Theorem \ref{thm:quasi_homo_twSp}. \end{definition}

We will use two representatives of this equivalence class of quasimorphisms. The first is defined using the complex determinant of the unitary part of the polar decomposition.

\begin{example} \label{ex:determinant_quasimorphism} \cite{bg1992} The \emph{determinant quasimorphism} $r:\twSp(2n) \to \R$ is the lift of the composition
\begin{equation} \Sp(2n) \xrightarrow{U} \U(n) \xrightarrow{\text{det}_\C} \U(1) \simeq \R/\Z \end{equation} 
\end{example}

\begin{lemma} \cite{bg1992,d1979} \label{lem:determinant_quasimorphism} There exists a constant $C > 0$ such that
\[
|r(\tilde{\Phi}\tilde{\Psi}) - r(\tilde{\Phi}) - r(\tilde{\Psi})| < C \qquad\text{for all}\qquad \tilde{\Phi},\tilde{\Psi} \in \twSp(2n)
\]
\end{lemma}

\noindent The second uses the eigenvalues of $A$, and appears in formulations of the Conley-Zehnder index.

\begin{example} \cite{g2014} \label{ex:eigenvalue_quasimorphism} The \emph{eigenvalue quasimorphism} $e:\twSp(2n) \to \R$ is the lift of the map
\[\underline{e}:\Sp(2n) \to U(1)\]
defined as follows. Let $A$ be a symplectic matrix. For each eigenvalue $\lambda \in U(1) \setminus \{\pm 1\}$ with generalized complex eigenspace $V(\lambda)\subset \C^{2n}$, consider the real quadratic form
\[Q(A,\lambda):V(\lambda) \otimes V(\lambda) \to \R \qquad \text{given by}\qquad Q(A,\lambda)(v \otimes w) = \on{Im}(\omega(v,\bar{w}))\]
Let $m(A,\lambda)$ be the maximal real dimension of a real subspace of $V(\lambda)$ on which $Q(A,\lambda)$ is positive definite. Finally, let $n(A)$ denote the sum of the complex dimensions of the generalized eigenspaces of $A$ with negative real eigenvalues. Then
\begin{equation} \label{eqn:eigenvalue_quasimorphism}
\underline{e}(A) = (-1)^{n(A)/2} \cdot \prod_{\lambda \in U(1) \setminus \{\pm 1\}} \lambda^{m(A,\lambda)/2} \in U(1)
\end{equation}
Note that if $A$ has no eigenvalues in $(U(1) \cup (-\infty,0)) \setminus \{1\}$, then $\underline{e}(A)$ is $1$ by convention. \end{example}

\begin{prop}[Trace Estimate] \label{prop:trace_estimate_r} Let $A:[0,T] \to \Sp(2n)$ be a path of symplectic matrices with $A_0 = \on{Id}$ and let
\[
S_t := -\Omega\frac{dA}{dt} A_t^{-1}
\]
Suppose that $S$ is positive semi-definite. Then the determinant quasimorphism $r$ (see Example \ref{ex:determinant_quasimorphism}) satisfies 
\[
r(A) \le \frac{8n^2}{\pi} \int_0^T \on{tr}(S_t) dt
\]
where we regard $A$ as an element of the universal cover $\widetilde{\on{Sp}}(2n)$.
\end{prop}

\begin{proof} Let $U$ be the unitary part of $A$. The rotation quasimorphism on $A$ is given by
\[r(A) = \frac{1}{2\pi i} \cdot \int_0^T \frac{\frac{d}{dt}(\text{det}_\C(U_t))}{\text{det}_\C(U_t)} dt = \frac{1}{2\pi i} \cdot \int_0^T \text{tr}_\C(\frac{dU_t}{dt} U_t^{-1}) dt\]
The trace estimate in Lemma \ref{lem:trace_estimate} implies that the trace above can be estimated as
\[
|\text{tr}_\C(\frac{dU_t}{dt} U_t^{-1})| \le 16n^2 \cdot \on{tr}(S_t) \qquad\text{and thus}\qquad |r(A)| \le \frac{8n^2}{\pi} \cdot \int_0^T \on{tr}(S_t) ds\qedhere
\]
\end{proof}

\subsection{Conley-Zehnder Index} \label{subsec:CZ_index} The Conley-Zehnder index is an invariant of certain non-degenerate elements of the universal cover of the symplectic group. 

\vspace{3pt}

Recall that a symplectic matrix $A \in \Sp(2n)$ is \emph{non-degenerate} if none of its eigenvalues is equal to $1$. That is
\[
\on{det}(A - \on{Id}) \neq 0
\]
We let $\Sp_\star(2n) \subset \Sp(2n)$ denote the open set of non-degenerate symplectic matrices and $\widetilde{\Sp}_\star(2n)$ denote its inverse image in the universal cover. 

\begin{thm} \label{thm:CZ_index} (cf. \cite{g2014}) There is a unique continuous map, called the Conley-Zehnder index, of the form
\[
\CZ:\widetilde{\Sp}_\star(2n) \to \Z \qquad \text{for each dimension $n$}
\]
that satisfies the following list of axioms.
\begin{itemize}
    \item[(a)] (Naturality) $\CZ$ is invariant under conjugation.
    \[\CZ(\widetilde{\Psi}\widetilde{\Phi}\widetilde{\Psi}^{-1}) = \CZ(\widetilde{\Phi})\]
    \item[(b)] (Direct Sum) $\CZ$ is additive under direct sum.
    \[\CZ(\widetilde{\Phi} \oplus \widetilde{\Psi}) = \CZ(\widetilde{\Psi}) + \CZ(\widetilde{\Phi})\]
    \item[(c)] (Maslov Index) If $\widetilde{\Psi} \in \pi_1(\Sp(2n))$ is an element of $\widetilde{\Sp}(2n)$ starting and ending on $\on{Id}$, then
    \[\CZ(\widetilde{\Psi}\widetilde{\Phi}) = 2\mu(\widetilde{\Psi}) + \CZ(\widetilde{\Phi})\]
    Here $\mu$ is the Maslov index of the loop (cf. \cite{ms2017}).
    \vspace{3pt}
    \item[(d)] (Signature) Let $\widetilde{\Psi} \in \widetilde{\Sp}_\star(2n)$ be the homotopy class of the path $t \mapsto \exp(2\pi \Omega At)$ for $t \in [0,1]$, where $A$ is a non-degenerate symmetric matrix with eigenvalues $\lambda$ of norm less than $1$. Then
    \[\CZ(\widetilde{\Psi}) = \frac{1}{2} \cdot \on{sign}(A) \qquad\text{where $\on{sign}(\cdot)$ is the signature.}\]
\end{itemize}
\end{thm}

\noindent There are a number of inequivalent ways to extend the Conley-Zehnder from $\widetilde{\Sp}_\star(2n)$ to $\widetilde{\Sp}(2n)$ \cite{g2014,gg2016}. We are primarily interested in the following extension.

\begin{definition} \label{def:LCZ_index} The \emph{lower semi-continuous Conley-Zehnder index} is the map
\[\LCZ:\widetilde{\Sp}(2n) \to \Z \qquad\text{with}\qquad \LCZ(\widetilde{\Phi}) = \inf\big\{\liminf_{j \to \infty} \CZ(\widetilde{\Phi}_j) \; : \; \widetilde{\Phi}_j \in \widetilde{\Sp}(2n) \text{ with }\widetilde{\Phi}_j \to \widetilde{\Phi}\big\}\]\end{definition}

Evidently, $\CZ$ extends $\LCZ$ in the sense that $\CZ = \LCZ$ on $\widetilde{\Sp}_\star(2n)$ and some axioms of $\CZ$ survive as properties of $\LCZ$. We record these properties, along with a key lower bound, below.

\begin{lemma} \label{lem:LCZ_index} The lower semi-continuous Conley-Zehnder index $\LCZ$ has the following properties.
\begin{itemize}
    \item[(a)] (Naturality) $\LCZ$ is invariant under conjugation.
    \[\LCZ(\widetilde{\Psi}\widetilde{\Phi}\widetilde{\Psi}^{-1}) = \LCZ(\widetilde{\Phi})\]
    \item[(b)] (Direct Sum) $\LCZ$ is additive under direct sum.
    \[\LCZ(\widetilde{\Phi} \oplus \widetilde{\Psi}) = \LCZ(\widetilde{\Psi}) + \CZ(\widetilde{\Phi})\]
    \item[(c)] (Maslov Index) If $\widetilde{\Psi} \in \pi_1(\Sp(2n))$ is an element of $\widetilde{\Sp}(2n)$ starting and ending on $\on{Id}$, then
    \[\LCZ(\widetilde{\Psi}\widetilde{\Phi}) = 2\mu(\widetilde{\Psi}) + \LCZ(\widetilde{\Phi})\]
    \item[(d)] (Lower Bound) Let $\rho:\widetilde{\Sp}(2n) \to \R$ be the homogeneous rotation quasi-morphism. Then
    \[\LCZ(\widetilde{\Phi}) \ge 2 \cdot \rho(\widetilde{\Phi}) - n\] 
\end{itemize}
\end{lemma}

\noindent The naturality and Maslov index properties follow immediately from the same properties of $\CZ$. The direct sum property is \cite[Lemma 4.3, p. 45]{gg2016} and the lower bound is given in \cite[Eq. 4.6, p. 43]{gg2016}. Note that the lower bound in \cite{gg2016} is stated in terms of the mean index $\hat{\mu}$ (see \cite[p. 41]{gg2016}).

\vspace{3pt}

As an example, we calculate $\LCZ$ in the case of paths in $U(1)$. We will use this calculation in \S \ref{sec:toric_formula}.

\begin{lemma} \label{lem:LCZ_of_U1_element} Let $\widetilde{u}(\theta) \in \widetilde{\Sp}(2)$ for $\theta \in \R$ be the homotopy class of the path
\[[0,1] \to \on{U}(1) \subset \Sp(2) \qquad\text{where}\qquad t \mapsto \on{exp}(2\pi i \theta \cdot t) \]
Then $\LCZ(\widetilde{u}(\theta))$ is given by $2 \lceil \theta \rceil - 1$. As a special case, we have $\LCZ(\widetilde{\on{Id}}_2) = -1$.
\end{lemma}
\begin{proof} By the signature property in Theorem \ref{thm:CZ_index}, we can directly compute that
\[
\CZ(\widetilde{u}(\theta)) = -1 \text{ if }\theta \in (-1,0) \qquad\text{and}\qquad \LCZ(\widetilde{u}(\theta)) = -1 \text{ if }\theta = 0
\]
Since we can write $\theta = \lceil \theta \rceil + r$ for $r \in (-1,0]$, the Maslov index property then implies that
\[\LCZ(\widetilde{u}(\theta)) = \LCZ(\widetilde{u}(\lceil \theta \rceil) \cdot \widetilde{u}(r)) = 2\mu(\widetilde{u}(\lceil \theta \rceil)) + \LCZ(\widetilde{u}(r)) = 2 \lceil \theta \rceil - 1 \qedhere\] \end{proof}

\subsection{Indices Of Orbits} \label{subsec:indices_of_orbits} We conclude this section by discussing the Conley-Zehnder index of Hamiltonian and Reeb orbits.  

\begin{definition} \label{def:LCZ_Hamiltonian} Let $X$ be a symplectic manifold with $c_1(TX) = 0$ and let $H:X \to \R$ be a Hamiltonian. The (lower semi-continuous) Conley-Zehnder index
\[
\LCZ(X,H;\gamma) \qquad \text{or simply} \qquad \LCZ(\gamma) 
\]
of a contractible periodic Hamiltonian orbit $\gamma:[0,T] \to X$ is defined as follows. Let $T\Phi$ be the differential of the Hamiltonian flow. Choose a disk $\Sigma \subset X$ bounded by $\gamma$ and a trivialization $\tau:TX|_\Sigma \simeq \C^n$. Let $\widetilde{\Psi}_\tau \in \widetilde{\Sp}(2n)$ be the homotopy class of the path
\begin{equation} \label{eqn:linearized_flow_Hamiltonian}
\Psi_\tau:[0,T] \to \Sp(2n) \qquad\text{given by}\qquad \Phi_\tau(t) = \tau_{\gamma(t)} \circ T\Phi(t,z) \circ \tau_{\gamma(0)}^{-1} \in \Sp(2n)
\end{equation}
We define $\LCZ(X,H;\gamma)$ to be $\LCZ(\widetilde{\Psi}_\tau)$. Since $c_1(TX) = 0$, this is independent of $\Sigma$. \end{definition}

\begin{definition} \label{def:LCZ_Reeb} Let $(Y,\xi)$ be a closed contact manifold with $c_1(\xi) = 0$ and let $\alpha$ be a contact form. The (lower semi-continuous) Conley-Zehnder index
\[
\LCZ(Y,\alpha;\gamma)\qquad\text{or simply}\qquad\LCZ(\gamma)
\]
of a contractible periodic Reeb orbit $\gamma:[0,T] \to Y$ is defined as follows. Let $T\Phi|_\xi$ be the differential of the Reeb flow restricted to $\xi$. Choose a disk $\Sigma$ bounded by $\gamma$ and a trivialization $\tau:\xi|_\Sigma \simeq \C^{n-1}$. Let $\widetilde{\Psi}_\tau \in \widetilde{\Sp}(2n-2)$ be the homotopy class of the path
\begin{equation} \label{eqn:linearized_flow_Reeb}
\Psi_\tau:[0,T] \to \Sp(2n) \qquad\text{given by}\qquad \Phi_\tau(t) = \tau_{\gamma(t)} \circ T\Phi|_\xi(t,y) \circ \tau_{\gamma(0)}^{-1}
\end{equation}
We define $\LCZ(Y,\alpha;\gamma)$ to be $\LCZ(\widetilde{\Psi}_\tau)$. Since $c_1(TX) = 0$, this is independent of $\Sigma$.   
\end{definition}

In the case of a Liouville domain, these two versions of $\LCZ$ can be related.

\begin{lemma} Let $(X,\lambda)$ be a Liouville domain with boundary $(Y,\xi)$. Fix an contractible loop
\[\gamma:[0,T] \to Y = \partial X \qquad\text{with}\qquad \gamma(0) = \gamma(T)\]
that is an orbit of the canonical Hamiltonian $H_X$, or equivalently a Reeb orbit of $\lambda|_Y$. Then
\[\LCZ(X,H_X;\gamma) = \LCZ(Y,\lambda|_Y;\gamma) - 1\] 
\end{lemma}

\begin{proof} Let $Z$ and $R$ denote the Liouville and Hamiltonian vector field, respectively. Also, we adopt the shortened notation $H = H_X$. Note that we have a splitting
\[TX = \xi^\omega \oplus \xi = \on{span}(Z,R) \oplus \xi \simeq \C \oplus \xi\]
Now choose a disk $\Sigma \subset Y$ bounded by $\gamma$ and let $\tau:\xi|_\Sigma \simeq \C^{n-1}$ be the unique isotopy class of trivialization of $\xi|_\Sigma$. Then we may form a trivialization $\sigma:TX|_\Sigma \simeq \C^n$ as the direct sum
\[\sigma = \tau_{\on{std}} \oplus \tau:TX \simeq \C \oplus \xi \simeq \C^n \qquad\text{where}\qquad \tau_{\on{std}}:\C \simeq \C \text{ is the tautological trivialization}\]
The flow $\Phi$ of $H$ preserves $Z$ and $R$. Indeed, $R$ generates $\Phi$, and since $ZH = H$ we have
\[
dH = \mathcal{L}_Z(dH) = \mathcal{L}_Z(\iota_R \omega) = \iota_{[Z,R]}\omega + \iota_R \mathcal{L}_Z\omega = \iota_{[Z,R]}\omega + dH \qquad\text{and so}\qquad [R,Z] = 0
\]
Thus the paths $\Psi_\sigma:[0,T] \to \Sp(2n)$ in (\ref{eqn:linearized_flow_Hamiltonian}) and $\Psi_\tau:[0,T] \to \Sp(2n-2)$ in $(\ref{eqn:linearized_flow_Reeb})$ are related by
\[
\Psi_\sigma = \on{Id}_2 \oplus \Psi_\tau \qquad\text{where}\qquad \on{Id}_2:[0,T] \to \Sp(2)\text{ is the constant path}
\]
By Lemma \ref{lem:LCZ_of_U1_element}, we have $\LCZ(\widetilde{\on{Id}}_2) = -1$. Thus we have
\[\LCZ(X,H;\gamma) = \LCZ(\widetilde{\Psi}_\sigma) = \LCZ(\widetilde{\on{Id}}_2) + \LCZ(\widetilde{\Psi}_\tau) = -1 + \LCZ(Y,\alpha;\gamma)\qedhere\]\end{proof}

As a corollary, we have a different characterization of dynamical convexity in terms of the Hamiltonian flow of the canonical Hamiltonian.

\begin{cor} \label{cor:dyn_cvx_Hamiltonian} A star-shaped domain $X \subset \C^n$ is dynamically convex if and only if the closed orbits $\gamma$ of the canonical Hamiltonian $H_X$ satisfy
\[\LCZ(X,H_X;\gamma) \ge n\]
\end{cor}

\section{Ruelle Density And Invariant} \label{sec:Ruelle_density} In this section, we construct the Ruelle invariant of a symplectic cocycle of a flow on a compact manifold, and demonstrate its basic properties.

\subsection{Rotation Function} \label{subsec:rotation_function} For the rest of the section, we fix a flow and a symplectic cocycle.
\[\phi:\R \times Y \to Y \qquad \Phi:\rho^*E \to \phi^*E\]
We also fix a $\phi$-invariant Borel measure $\mu$. Our construction of the Ruelle invariant requires an auxilliary family of maps
\[\tilde{u}_T:Y \to \R \qquad\text{also denoted by} \qquad \tilde{u}_T(\Phi;J,\tau):Y \to \R \qquad\text{for }T \in [0,\infty)\]
depending on a choice of complex structure $J$ and trivialization $\tau$. We refer to $\tilde{u}_T$ as the \emph{rotation function} at time $T$. The goal of this subsection is to define the rotation function and prove some basic properties.

\vspace{3pt}

Let us, first, recall the definitions of the various auxilliary data required to build $\tilde{u}_T$. 

\begin{definition} A (compatible) \emph{complex structure} $J$ on $E$ is an bundle map $J:E \to E$ such that
\[J^2 = -\on{Id} \qquad\text{and}\qquad \omega(J-,-) \text{ is a fiberwise inner product}\]
\end{definition}

A choice of compatible complex structure $J$ gives $E$ the structure of a Hermitian vector bundle. Standard results in algebraic topology (cf. \cite{ms2017}) state that the space $\mathcal{J}(E,\omega)$ of compatible complex structures on $E$ is contractible. Moreover, any two choices $J,J'$ of such complex structures yield isomorphic Hermitian vector bundles $(E,\omega,J) \simeq (E,\omega,J')$.

\begin{definition} The \emph{determinant bundle} $\Lambda E$ of $(E,J)$ is the maximal wedge power of $E$ as a complex vector bundle. That is
\[\Lambda E := \wedge^k\big(E\big) \qquad\text{where}\qquad k = \on{rank}_\C(E) \]
A \emph{trivialization} $\tau:\Lambda E \simeq \C$ is a unitary bundle map to the trivial bundle. \end{definition} 

The determinant bundle of $E$ is independent of $J$ up to (homotopically unique) isomorphism. In particular, the set of homotopy classes of trivialization
\[
\mathcal{T}(\Lambda E) := \{\text{trivializations }\tau:\Lambda E \simeq \C\} \; \big/ \text{ homotopy }
\]
is well-defined, without reference to a specific choice of $J$. The determinant bundle admits a trivialization if and only if $c_1(E) = c_1(\Lambda E) = 0$. Furthermore, the space of trivializations is naturally a torsor over $[Y,S^1] \simeq H^1(Y;\Z)$.

\vspace{3pt}

We are now ready to proceed with the construction of the rotation function.

\begin{construction} Choose a compatible complex structure $J$ on $E$ and an explicit unitary trivialization $\tau:\Lambda E \simeq \C$ in the chosen class. Start by taking the polar decomposition of $\Phi$
\[\Phi:\rho^*E \xrightarrow{P} \rho^*E \xrightarrow{U} \phi^*E\]
Here $P$ is self-adjoint and $U$ is unitary with respect to $\omega$ and $J$. The determinant $\Lambda U$ of $U$ and the trivialization $\tau$ determine a unitary bundle map
\[u:\C \xrightarrow{\tau^{-1}} \rho^*\Lambda E \xrightarrow{\Lambda U} \phi^*\Lambda E \xrightarrow{\tau} \C \quad\text{or equivalently a map}\quad u:\R \times Y \to U(1)\]
The map $u$ sends $0 \times Y$ to $1 \in U(1)$, and is therefore null-homotopic. Thus there is a unique lift
\[
\tilde{u}:\R \times Y \to \R \qquad \text{such that}\qquad \exp(2\pi i \cdot \tilde{u}) = u \quad\text{and}\quad \tilde{u}|_{0 \times Y} = 0
\]
The rotation function $\tilde{u}_T$ is simply this lift at time $T$, i.e. $\tilde{u}(T,-)$. \end{construction}

\vspace{3pt}

We may view these maps as a version of the rotation quasimorphism applied pointwise in $Y$ to $\Phi$. To make this precise, it will be helpful to fix some notation.

\begin{notation}\label{notation:lift_cocycle} Given a trajectory $\gamma:[0,T] \to Y$ of $\phi$ and a trivialization $\Xi:E|_\gamma \simeq \C^n$, we let
\[
\Phi_\Xi:[0,T] \to \on{Sp}(2n) \qquad\text{denote the map}\qquad \Phi_\Xi(t) := \Xi(t) \circ \Phi(t,x) \circ \Xi(0)^{-1}
\]
Furthermore, we let $\tilde{\Phi}_\Xi$ denote the unique lift of $\Phi_\Xi$ to the universal cover satisfying $\tilde{\Phi}_\Xi(0) = \tilde{\on{Id}}$. 
\end{notation}

\begin{lemma}[Quasimorphism] \label{lem:rotation_function_vs_rotation_quasimorphism} Let $\gamma:[0,T] \to Y$ be a trajectory of $\phi$ with $\gamma(0) = X$ and let $\Xi:E|_\gamma \simeq \C^n$ be a unitary trivialization of $E$ over $\gamma$ such that the map $\Lambda\Xi:\Lambda E|_\gamma \simeq \C$ is $\tau|_\gamma$. Then
\[
\tilde{u}_T(x) = r(\widetilde{\Phi}_\Xi)
\]
\end{lemma}

\begin{proof} Since $\Xi$ is unitary, the unitary parts $U_\Xi$ of $\Phi_\Xi$ and $U$ of $\Phi$ are related by
\[
U_\Xi(t) := \Xi(t) \circ U(t,x) \circ \Xi(0)^{-1} \in \Sp(2n)
\]
The map on the determinant bundle $\Lambda \C^n = \C$ induced by $U_\Xi$ is simply the determinant over $\C$. In particular, we have
\[
\on{det}_\C(U_\Xi(t)) =  \Lambda \Xi(t) \circ \Lambda U(t,x) \circ \Lambda \Xi(0)^{-1} = \tau(\phi(t,x)) \circ \Lambda U(t,x) \circ \tau^{-1}(x) = u(t,x)
\]
In particular, the maps $r \circ \widetilde{\Phi}_\Xi:[0,T] \to \R$ and $\tilde{u}(x):[0,T] \to \R$ are both lifts of the same map $[0,T] \to U(1)$ that are $0$ at $t = 0$. This implies that they agree, proving the result. \end{proof}

The rotation functions at time $T$ essentially define a sub-additive process in the sense of Kingman \cite{k1973}. We use the following definition, which specializes the one in \cite{k1973} to our setting.

\begin{definition} \label{def:subadditive_process} A \emph{sub-additive process} $f_T:Y \to \R$ for $T \in [0,\infty)$ for the dynamical system $(Y,\phi)$ with invariant measure $\mu$ is a family of $\mu$-integrable functions that, for some $C > 0$, satisfy
\[f_{S+T} \le f_S + \phi^*_Sf_T + C \qquad\qquad \int_Y f_T\cdot \mu \ge -C \cdot T \qquad\qquad \int_Y \big(\sup_{0 \le S \le 1} |f_S|\big) \cdot \mu < \infty\]
\end{definition}

\begin{lemma} \label{lem:subadditive_process} The family of maps $\tilde{u}_T$ are a sub-additive process for $(Y,\phi)$ and $\mu$.
\end{lemma}

\begin{proof} We verify the properties in Definition \ref{def:subadditive_process}. For the first property, fix a trajectory $\gamma:[0,S+T] \to Y$ of $\phi$ with $\gamma(0) = x$, and choose a unitary trivialization $\Xi:E|_\gamma \simeq \C^n$ inducing the trivialization $\tau:\Lambda E \simeq \C$. Define
\begin{equation} \label{eqn:lem:subadditive_process_1}
\Phi_\Xi:[0,S+T] \to \Sp(2n) \qquad \Phi_\Xi(t) := \Xi(t) \circ \Phi(t,x) \circ \Xi(0)^{-1}
\end{equation}
Let $\tilde{\Phi}_\Xi:[0,S+T] \to \twSp(2n)$ denote the lift to the universal cover. Then by Lemma \ref{lem:rotation_function_vs_rotation_quasimorphism} and the quasimorphism property of $r$, there is a constant $c > 0$ such that
\begin{equation} \label{eqn:lem:subadditive_process_2} \tilde{u}_{S+T}(x) = r(\tilde{\Phi}_\Xi(S+T)) \le r(\tilde{\Phi}_\Xi(S)) + r(\tilde{\Phi}_\Xi(S+T)\tilde{\Phi}_\Xi(S)^{-1}) + C\end{equation}
Clearly $\tilde{u}_S(x) = r(\tilde{\Phi}_\Xi(S))$ by Lemma \ref{lem:rotation_function_vs_rotation_quasimorphism}. Moreover, the cocycle property of $\Phi$ implies that
\[
\Phi_\Xi(S+t)\Phi_\Xi(S)^{-1} =  \Xi(t) \circ \Phi(S + t,x) \circ \Phi(S,x) \circ \Xi(s) = \Xi(S + t) \circ \Phi(t,\phi_S(t)) \circ \Xi(S)^{-1}
\]
Thus Lemma \ref{lem:rotation_function_vs_rotation_quasimorphism} also implies that $\phi^*_S\tilde{u}_T(x) = r(\tilde{\Phi}_\Xi(S+T)\tilde{\Phi}_\Xi(S)^{-1})$. The first property in Definition \ref{def:subadditive_process} then follows from (\ref{eqn:lem:subadditive_process_2}). To see the second property, note that if $T = m + S$ for $S \in [0,1]$, we have
\[
\int_Y \tilde{u}_T \cdot \mu \ge \sum_{k=0}^{m-1} \int_Y \phi^*_k\tilde{u}_1 \cdot \mu + \int_Y \phi^*_m\tilde{u}_S \cdot \mu - cm \ge (-c + M \cdot \mu(Y)) \cdot T
\]
Here $M$ is the minimum of $\tilde{u}_T$ for $T \in [0,1]$ and $0$. We can thus take the constant in the lemma to be $c - M \cdot \mu(Y)$. Finally, the third property follows immediately from the fact that $\tilde{u}:[0,1] \times Y \to \R$ is continuous and $Y$ is compact.
\end{proof}

In \cite{k1973}, Kingman proves several ergodic theorems, one of which can be stated as follows.

\begin{thm}[\cite{k1973}, Thm 4] \label{thm:kingman_ergodic} Let $f_T$ be a sub-additive process in the sense of Definition \ref{def:subadditive_process}. Then $\frac{f_T}{T}$ converges in $L^1(Y,\mu)$ and pointwise almost everywhere as $T \to \infty$.
\end{thm}{}

Applying Theorem \ref{thm:kingman_ergodic} to $\tilde{u}$ via Lemma \ref{lem:subadditive_process}, we immediately acquire the following result.

\begin{cor} The family of maps $\frac{\tilde{u_T}}{T}$ converges in $L^1(Y,\mu)$ and pointwise almost everywhere as $T \to \infty$. \end{cor}

Critically, the limit of $\frac{\tilde{u}_T}{T}$ is independent of the auxilliary choices made. To demonstrate this, we need the following lemma.

\begin{lemma}[Automorphism] \label{lem:u_process_automorphism} There is a constant $C > 0$ with the property that, if $\Psi:Y \to \on{Aut}(E)$ is a symplectic bundle map homotopic to $\on{Id}$, then
\[|\tilde{u}_T(\Psi^*\Phi;J,\tau) - \tilde{u}_T(\Phi;J,\tau)| \le C\]
\end{lemma}

\begin{proof} Let $\tilde{\Psi}:Y \to \widetilde{\on{Aut}}(E)$ denote any lift of $\Psi$ to the (fiberwise) universal cover bundle $\widetilde{\on{Aut}}(E)$ of $\on{Aut}(E)$.  Fix a trajectory $\gamma:[0,T] \to Y$ of $\phi$ with $\gamma(0) = x$, and choose a unitary trivialization $\Xi:E|_\gamma \simeq \C^n$ with $\Lambda \Xi = \tau$. Let $\Phi_\Xi$ and $\Psi^*\Phi_\Xi$ be defined as in Notation \ref{notation:lift_cocycle}, and let
\[
\Psi_\Xi:[0,T] \to \on{Sp}(2n) \quad\text{denote}\quad \Psi_\Xi(t) := \Xi(t) \circ \Psi(\gamma(t)) \circ \Xi(t)^{-1}
\]
Note that $\Phi_\Xi$, $\Psi^*\Phi_\Xi$ and $\Psi_\Xi$ are all related by the following identity.
\begin{equation} \label{eqn:lem:u_process_automorphism_1}
\Psi^*\Phi_\Xi(t) = \Psi_\Xi(t) \Phi_\Xi(t) \Psi_\Xi(t)^{-1}
\end{equation}
The trivialization induces a bundle isomorphism $\gamma^*\widetilde{\on{Aut}}(E) \simeq \twSp(2n)$, and thus the lift $\tilde{\Psi}$ of $\Psi$ induces a unique lift $\tilde{\Psi}_\Xi$ of $\Psi_\Xi$. The identity (\ref{eqn:lem:u_process_automorphism_1}) lifts to
\begin{equation} \label{eqn:lem:u_process_automorphism_2}
\widetilde{\Psi^*\Phi}_\Xi(t) = \tilde{\Psi}_\Xi(t) \tilde{\Phi}_\Xi(t) \tilde{\Psi}_\Xi(t)^{-1}
\end{equation}
Indeed, it suffices to check (\ref{eqn:lem:u_process_automorphism_2}) at $t = 0$, where both sides are $\tilde{\on{Id}} \in \twSp(2n)$.

\vspace{3pt}

To acquire the desired conclusion from (\ref{eqn:lem:u_process_automorphism_2}), we note that by Lemma \ref{lem:rotation_function_vs_rotation_quasimorphism}, we have
\begin{equation} \label{eqn:lem:u_process_automorphism_3}
\tilde{u}_T(\Psi^*\Phi;J,\tau) = r \circ \widetilde{\Psi^*\Phi}_\Xi(T) \qquad\text{and}\qquad \tilde{u}_T(\Phi;J,\tau) = r \circ \widetilde{\Phi}_\Xi(T)\end{equation}
On the other hand, let $c > 0$ be constant in Lemma \ref{lem:determinant_quasimorphism}. Then
\[|r \circ \tilde{\Psi}_\Xi + r \circ \tilde{\Psi}^{-1}_\Xi| \le |r(\tilde{\on{Id}})| + c = c\]
Therefore, at time $T$ we have the following inequality.
\[|r(\widetilde{\Psi^*\Phi}_\Xi(T)) - r(\tilde{\Phi}_\Xi(T))| \le |r(\widetilde{\Phi}_\Xi(T)) + r(\tilde{\Psi}_\Xi) + r(\tilde{\Psi}^{-1}_\Xi) - r(\tilde{\Phi}_\Xi(T))| + 2c \le 3c\]
The result now follows from (\ref{eqn:lem:u_process_automorphism_3}) by taking $C = 3c$. \end{proof}

\begin{prop} The limit of $\frac{\tilde{u}_T}{T}$ as $T \to \infty$ is independent of $J$ and the choice of representative of $\tau$. 
\end{prop}

\begin{proof} For convenience, we fix the following notation for this proof.
\[g(\Phi,J,\tau) := \lim_{T \to \infty} \frac{\tilde{u}_T(\Phi,J,\tau)}{T} \in L^1(Y,\mu;\R)\]

\vspace{3pt}

To show that the limit depends only on the isotopy class of $\tau$, let $\sigma:\Lambda E \simeq \C$ and $\tau:\Lambda E \simeq \C$ be isotopic unitary trivializations. Then we have
\[\sigma = f \tau \qquad\text{where}\qquad f:X \to U(1) \qquad \text{satisfies}\qquad [f] = 0 \in [Y,U(1)]\]
Since $f$ is null-homotopic, $f$ admits a lift $\tilde{f}:Y \to \R$ via the covering map $\exp(2\pi i \cdot):\R \to U(1)$. We can then relate $u_T(\Phi,J,\tau)$ and its lift to $u_T(\Phi,J,\sigma)$ by the following formulas.
\begin{equation} \label{prop:Ruelle_density:eq:1} u_T(\Phi,J,\sigma) = (f \circ \phi) u_T(\Phi,J,\tau) f^{-1} \qquad\text{and}\qquad \tilde{u}_T(\Phi,J,\sigma) = \tilde{f} \circ \phi + \tilde{u}_T(\Phi,J,\tau) - \tilde{f}\end{equation}
The first formula in (\ref{prop:Ruelle_density:eq:1}) follows directly from the definition, while the second follows from the uniqueness of the lift that is $0$ along $0 \times Y$. We then see that
\[\|g(\Phi,J,\tau) - g(\Phi,J,\sigma)\|_{L^1} = \lim_{T \to \infty} \frac{1}{T} \|\tilde{u}_T(\Phi,J,\sigma) - \tilde{u}_T(\Phi,J,\tau)\|_{L^1}\]
\[= \lim_{T \to \infty} \frac{1}{T} \| \tilde{f} \circ \phi_t - \tilde{f}\|_{L^1} \le  \lim_{T \to \infty} \frac{2 \|\tilde{f}\|_{L^1}}{T} = 0
\] 
Thus $g(\Phi,J,\tau) = g(\Phi,J,\sigma)$ in $L^1(Y,\mu;\R)$ and the limit depends only the class of $\tau$.

\vspace{3pt}

To prove independence of $J$, let $I$ and $J$ be two choices of compatible complex structure on $E$. There is a unitary bundle isomorphism
\[\Psi:Y \to \on{Aut}(E) \qquad\text{such that}\qquad \Psi^*\omega = \omega \qquad \Psi^*J = I \quad\text{and}\quad \Psi \sim \on{Id}_E\]
Here $\Psi$ is homotopic to the identity through symplectic bundle automorphisms. In particular, $\Psi^*\tau = \tau$ for any trivialization class $\tau$. Since the limit depends only on the trivialization homotopy class, we thus have
\[g(\Phi,J,\tau) = g(\Psi^*\Phi,\Psi^*J,\Psi^*\tau) = g(\Psi^*\Phi,I,\tau)\]
Using this identity and Lemma \ref{lem:u_process_automorphism}, we compute
\[\|g(\Phi,J,\tau) - g(\Phi,I,\tau)\|_{L^1} \le \lim_{T \to \infty} \frac{1}{T}\|\tilde{u}_T(\Psi^*\Phi,I,\tau) - \tilde{u}_T(\Phi,I,\tau)\|_{L^1} \le \lim_{T \to \infty} \frac{C \cdot \mu(Y)}{T} = 0\]
This proves that the limit is independent of $J$ and concludes the proof.\end{proof}

\subsection{Construction Of Invariant} \label{subsec:construction_of_ruelle} We are now ready to give a precise definition of the Ruelle density and invariant. Choose a complex structure $J$ and trivialization in class $\tau$, as in \S \ref{subsec:rotation_function}.

\begin{definition} The \emph{Ruelle density} $\on{ru}(\Phi,\tau)$ and the \emph{Ruelle invariant} $\on{Ru}(\Phi,\tau)$ are defined by
\[
\on{ru}(\Phi,\tau) := \lim_{T \to \infty} \frac{\tilde{u}_T}{T} \qquad\text{and}\qquad \on{Ru}(\Phi,\tau) := \int_Y \on{ru}(\Phi,\tau) \cdot \mu
\]
\end{definition}

\begin{prop} \label{prop:Ruelle_invariant} The Ruelle density and the Ruelle invariant satisfy the following formal properties.
\begin{itemize}
    \item[(a)] (Covariance) If $\Psi:(E,\Phi) \to (E',\Phi')$ is a symplectic cocycle isomorphism that maps $\tau$ to $\tau'$, then
    \[\on{ru}(\Phi,\tau) = \on{ru}(\Phi',\tau') \qquad\qquad \on{Ru}(\Phi,\tau,\mu) = \on{Ru}(\Phi',\tau',\mu)\]
    \item[(b)] (Direct Sum) If $\Phi = \Phi_1 \oplus \Phi_2$ is a direct sum of symplectic cocycles and $\tau = \tau_1 \otimes \tau_2$, then
    \[\on{ru}(\Phi_1 \oplus \Phi_2,\tau_1 \otimes \tau_2) = \on{ru}(\Phi_1,\tau_1) + \on{ru}(\Phi_2,\tau_2)\]
    \item[(c)] (Linearity) If $a\mu + b\nu$ is a positive combination of $\phi$-invariant Borel measures $\mu$ and $\nu$, then
    \[
    \on{Ru}(\Phi,\tau,a\mu + b\nu) = a\on{Ru}(\Phi,\tau,\mu) + b\on{Ru}(\Phi,\tau,\nu)\]
    \item[(d)] (Trivial Bundle) If $\Phi$ is a symplectic cocycle on $\C^n$ with the tautological trivialization $\tau_{\on{std}}$, then
    \[\on{ru}(\Phi,\tau_{\on{std}}) = \lim_{T \to \infty} \frac{q \circ \widetilde{\Phi}_T}{T} \qquad\qquad \on{Ru}(\Phi,\tau_{\on{std}}) = \lim_{T \to \infty} \frac{1}{T} \int_Y q \circ \widetilde{\Phi}_T \cdot \mu\]
    Here $q$ is any rotation quasimorphism (see \S \ref{subsec:rotation_quasimorphisms}) and $\widetilde{\Phi}:\R\times Y \to \widetilde{\on{Sp}}(2n)$ is the lift of $\Phi$ (regarded as a map $\R \times Y \to \on{Sp}(2n)$) to the universal cover $\widetilde{\on{Sp}}(2n)$.
\end{itemize}
\end{prop}

\begin{proof} These properties are more or less immediate from the properties of $\tilde{u}_T$. We discuss each proof separately below.

\vspace{3pt}

\noindent {\bf Covariance.} This is immediate since we can assume (by choice of $J'$ and $\tau'$) that $\Psi$ is unitary. 

\vspace{3pt}

\noindent {\bf Direct Sum.} Choose explicit complex structures $J_i$ and unitary trivializations $\tau_i:\Lambda E_i \simeq \C$. We adopt the notatation
\[E = E_1 \oplus E_2 \qquad \Phi = \Phi_1 \oplus \Phi_2 \qquad J = J_1 \oplus J_2 \qquad \tau = \tau_1 \otimes \tau_2\]
The unitary part $U$ of the cocycle $\Phi$ with respect to $J$ and the determinant $\Lambda U$ can be written in terms of the unitary parts $U_i$ of $\Phi_i$ as
\[U = U_1 \oplus U_2 \qquad \Lambda U = \Lambda U_1 \otimes \Lambda U_2\]
Therefore, the induced maps $Y \to U(1)$ satisfy the following identities.
\[u(\Phi,J,\tau) = u(\Phi_1,J_1,\tau_1)u(\Phi_2,J_2,\tau_2) \qquad\text{and}\qquad \tilde{u}(\Phi,J,\tau) = \tilde{u}(\Phi_1,J_1,\tau_1) + \tilde{u}(\Phi_2,J_2,\tau_2)\]
The additivity of the Ruelle density and invariant now follows directly from the definition.

\vspace{3pt}

 \noindent {\bf Linearity.} This follows from the fact that $\on{ru}(\Phi,\tau) \in L^1(Y,\mu) \cap L^1(Y,\nu)$ and the linearity of integration against measures. 

\vspace{3pt}

\noindent {\bf Trivial Bundle.} Clearly, it suffices to prove the result for the rotation quasimorphism $r$. Let $\tilde{\Phi}:\R \times Y \to \twSp(2n))$ denote the lift of $\Phi:\R \times Y \to \Sp(2n)$ to the universal cover. Let $\Xi:\C^n \to \C^n$ be the identity trivialization on $\C^n$, so that
\[\Phi_\Xi = \Phi \qquad\text{and}\qquad \Lambda\Xi = \tau_{\on{std}}\]
Then by Lemma \ref{lem:rotation_function_vs_rotation_quasimorphism}, we have
\[\tilde{u}_T(x) = r \circ \tilde{\Phi}(T,x)\]
The result now follows immediately from the definition of $\on{ru}$ and $\on{Ru}$. 
\end{proof}

As an easy consequence of Proposition \ref{prop:Ruelle_invariant}(d) and Proposotion \ref{prop:trace_estimate_r}, we acquire a key trace bound on the Ruelle invariant.  

\begin{lemma}[Trace Bound] \label{lem:trace_bound_Ru} Let $\Phi$ be a symplectic cocycle on $\C^n$ generated by a map $A:Y \to \mathfrak{sp}(2n)$. That is
\[\frac{d}{dt}\big(\Phi(t,x)\big) = A(\phi(t,x)) \circ \Phi(t,x)\]
Assume that $-\Omega A$ is positive semi-definite, where $\Omega$ is the matrix representing the standard symplectic form. Then
\[\on{Ru}(\Phi,\tau_{\on{std}}) \le \frac{8n^2}{\pi} \cdot \int_Y \on{tr}(-\Omega A) \cdot \mu\]
\end{lemma}

\begin{proof} By Proposition \ref{prop:Ruelle_invariant}(d) and Proposition \ref{prop:trace_estimate_r}, we know that
\begin{equation} \label{eqn:trace_bound_Ru_1}
\on{Ru}(\Phi,\tau,\mu) = \lim_{T \to \infty} \frac{1}{T} \int_Y r \circ \widetilde{\Phi}_T \cdot \mu \le \frac{8n^2}{\pi}\lim_{T \to \infty} \frac{1}{T} \int_Y \Big( \int_0^T \on{tr}(-\Omega A(\phi(t,x))) \cdot dt \Big) \mu
\end{equation}
Rearranging the order of integration and using the fact that $\phi$ is measure preserving, we see that
\[
\int_Y \Big(\int_0^T \on{tr}(-\Omega A(\phi(t,x))) \cdot dt \Big) \mu = \int_0^T \Big(\int_Y \phi^*_t\on{tr}(-\Omega A) \cdot \mu \Big) dt = \int_0^T \Big(\int_Y \on{tr}(-\Omega A) \cdot \mu\big) \cdot dt
\]
Therefore, the right hand side of (\ref{eqn:trace_bound_Ru_1}) simplifies to
\[
\frac{8n^2}{\pi}\lim_{T \to \infty} \frac{1}{T} \int_Y \Big( \int_0^T \on{tr}(-\Omega A(\phi(t,x))) \cdot dt \Big) \mu = \frac{8n^2}{\pi}\lim_{T \to \infty} \frac{1}{T} \int_0^T \Big(\int_Y \on{tr}(-\Omega A) \cdot \mu\big) \cdot dt = \frac{8n^2}{\pi}\int_Y \on{tr}(-\Omega A) \cdot \mu \qedhere
\]
\end{proof}

\section{Ruelle Bound For Convex Domains} \label{sec:Ruelle_invariant_of_convex_domains} In this section, we prove that the Ruelle invariant of a convex, star-shaped domain $X$ obeys the systolic inequality in Theorem \ref{thm:main_inequality}.  

\vspace{3pt}

The majority of our proof involves the analysis of a certain Laplacian integral on a Riemannian manifold admitting a nice, free $\R$-action. We carry out this analysis in \S \ref{subsec:linear_tensor_fields} and \S \ref{subsec:laplacian_functional}. We then discuss standard symplectic ellipsoids in \S \ref{subsec:ellipsoids}, before proceeding to the main proof in \S \ref{subsec:proof_of_main_inequality}.

\subsection{Linear Tensor fields} \label{subsec:linear_tensor_fields} We start by discussing linear tensor fields, i.e. tensor fields on a (Riemannian) manifold that are conformal with respect to a vector field. Let $M$ be a manifold.

\begin{definition} A vector field $V$ is \emph{cylindrical} if there is a codimension $1$ submanifold $Y \subset M$ such that $V$ is transverse to $Y$ and the flow $\Phi$ by $V$ defines a diffeomorphism
\[
\Phi:\R \times Y \simeq M
\] 
A \emph{cylindrical domain} $X \subset M$ is a codimension $0$ submanifold with boundary such that flow by $V$ defines a diffeomorphism
\[
X \simeq (-\infty,0] \times \partial X
\]\end{definition}

\begin{definition} A tensor field $\Psi$ on $M$ is \emph{$V$-linear} of slope $a \in \R$ if
\[\mathcal{L}_V\Psi = a \cdot \Psi\] 
\end{definition}

We will need some elementary properties of linear tensor fields, which we record in the following lemma. The proofs are simple and left to the reader.

\begin{lemma} \label{lem:properties_of_linear_tensors} $V$-linear tensor fields on $M$ have the following properties.
\begin{itemize}
    \item[(a)] (Linearity) If $\Phi$ and $\Psi$ are $V$-linear tensor fields of slope $a$ and $c$ is a constant, then
    \[\Phi + \Psi \qquad\text{and}\qquad c \cdot \Psi \qquad\text{are $V$-linear of slope $a$}\]
    \item[(b)] (Tensor Product) If $\Phi$ and $\Psi$ are $V$-linear tensor fields of slope $a$ and $b$, respectively, then
    \[\Phi \otimes \Psi \qquad\text{is $V$-linear of slope $a+b$}\]
    \item[(c)] (Integral) If $\mu$ is a $V$-linear volume form of slope $a > 0$ and $X$ is a cylindrical domain, then
    \[\int_{\partial X} \iota_V\mu = a \cdot \int_X \mu\]
    \item[(d)] (Derivative) If $\theta$ is a $V$-linear differential form of slope $a$, then
    \[d\theta \qquad\text{is $V$-linear of slope $a$}\]
\end{itemize}
\end{lemma}

We will be primarily interested in $V$-linear tensors in the presence of a metric. Fix the data of
\[\text{a $V$-linear metric $g$ of slope $1$}\]
To start, we note that $V$ is compatible with the covariant derivative and metric volume.

\begin{lemma}[Covariant Derivative] \label{lem:covariant_derivative_linear} The covariant derivative $\nabla$ of the metric $g$ satisfies 
\[
\mathcal{L}_V(\nabla\Psi) = \nabla(\mathcal{L}_V\Psi) \qquad\text{for any tensor field} \qquad\Psi
\]
\[
\langle\nabla_UV,W\rangle + \langle U,\nabla_WV\rangle = \frac{1}{2}\langle U,W\rangle \qquad\text{for any pair of vector-fields}\qquad U,W
\]
Thus $\nabla \Psi$ is $V$-linear of slope $a$ if $\Psi$ is $V$-linear of slope $a$. \end{lemma} 

\begin{proof} For the first formula, let $\Phi:\R \times M \to M$ be the flow of $V$. Then
\[\Phi_t^*g = e^{t}g\]
Metrics differing by a constant conformal factor have identical covariant derivatives. Therefore
\[\mathcal{L}_V(\nabla \Psi) = \frac{d}{dt}(\Phi^*_t(\nabla \Psi))|_{t=0} = \frac{d}{dt}(\nabla(\Phi^*_t\Psi))|_{t=0} = \nabla \frac{d}{dt}(\Phi^*_t\Psi)|_{t=0} = \nabla(\mathcal{L}_V\Psi)\]

For the second formula, let $U$ and $W$ be arbitrary $V$-linear vector fields of slope $0$. Since the metric connection is torsion free, $U$ and $W$ satisfy
\[
\nabla_UV = \nabla_V U + [U,V] = \nabla_V U \qquad \text{and}\qquad \nabla_WV = \nabla_V W + [W,V] = \nabla_V W
\]
Moreover, $\langle U,W\rangle$ is slope $1$ since $U$ and $W$ are slope $0$. Thus we have
\[\langle U,W\rangle = \nabla_V\langle U,W\rangle = \langle \nabla_VU,W\rangle + \langle U,\nabla_VW\rangle = \langle \nabla_UV,W\rangle + \langle U,\nabla_WV\rangle\]
Since $U$ and $W$ are arbitrary, this formula is satisfied fiberwise on $TM$, i.e. for all vector fields. \end{proof}

\begin{lemma}[Volume Form] \label{lem:V_linear_volume} The metric volume form $\mu_g$ of $g$ is $V$-linear of slope $\frac{\on{dim}(M)}{2}$. \end{lemma}

\begin{proof} We briefly adopt the notation $\mu_g = \mu(g)$. Consider the flow $\Phi$ of $V$, and note that
\[
\Phi^*_t\mu(g) = \mu(\Phi^*_tg) = \mu(e^{t}g) = e^{t\on{dim}(M)/2} \cdot \mu(g)
\]
Taking the derivative at $t = 0$ yields the desired result. \end{proof}

As an immediate corollary of Lemma \ref{lem:covariant_derivative_linear}, we note that the gradient, divergence and Laplacian of a tensor are all $V$-linear.

\begin{cor} \label{cor:grad_div_laplacian_linear} Let $F$ and $U$ be a $V$-linear function and vector field, both of slope $a$. Then
\[
\nabla F \qquad \on{div}(U) \qquad \Delta F \qquad \text{are all $V$-linear of slope }a - 1
\]\end{cor}

We will also need the following lemma in the next section.

\begin{lemma} \label{lem:V_linear_pos_def} Let $H$ be a $V$-linear function of slope $1$ with positive semi-definite Hessian $\nabla^2H:TM \to TM$ and suppose that $\nabla V:TM \to TM$ is self-adjoint. Then
\[H \cdot \Delta H \ge \frac{1}{2}|\nabla H|^2\]\end{lemma}

\begin{proof} Note that $\langle \nabla H,V\rangle = \mathcal{L}_VH = H$ since $H$ is slope $1$. Therefore, we can compute that for any vector-field $W$, we have
\[\langle W,\nabla H\rangle = \nabla_{W} H = \nabla_{W}\langle \nabla H,V\rangle = \langle W,\nabla^2H(V)\rangle + \langle W,\nabla_{\nabla H}V\rangle\]
The self-adjoint part of $\nabla V$ is $\frac{1}{2} \cdot \on{Id}$ by Lemma \ref{lem:covariant_derivative_linear}. Since $\nabla V$ is assumed to be self-adjoint, we thus conclude that
\[\langle W,\nabla H\rangle = \langle W,\nabla^2H(V)\rangle + \frac{1}{2} \langle W,\nabla H\rangle \text{ for all }W \qquad\text{and thus}\qquad 2 \cdot \nabla^2H(V) = \nabla H\]
Finally, note that if $\nabla^2 H$ is positive definite, then we know that
\[\langle \nabla^2H(U),\nabla^2H(U)\rangle \le \on{tr}(\nabla^2H) \cdot \langle \nabla^2H(U),U\rangle = \Delta H \cdot \langle \nabla^2H(U),U\rangle \qquad\text{for any vector field }U\]
 Applying this inequality to the Hessian $\nabla^2 H$ and the formula $2 \cdot \nabla^2 H(V) = \nabla H$, we find that
\[|\nabla H|^2 \le 4 \cdot |\nabla^2 H(V)|^2 \le 4 \cdot \Delta H \cdot \langle V,\nabla ^2 H(V)\rangle = 2 \cdot \Delta H \cdot \langle V,\nabla H\rangle = 2 H \cdot \Delta H \qedhere\]
\end{proof}

\subsection{Laplacian Functional} \label{subsec:laplacian_functional} Let $(M,g)$ be a Riemannian manifold with a cylindrical vector field $V$ such that $g$ is $V$-linear of slope $1$. Consider the space of $V$-linear functions
\[
\Gamma(M;V) := \{h \in C^\infty(M;\R) \; : \; Vh = h \}
\]
There is a convex open subset $U(M;V) \subset \Gamma(M;V)$ consisting of positive functions.
\[
U(M;V) := \{H \in \Gamma(M;V) \; : \; H > 0 \}
\]
Note that the sub-level set $X = H^{-1}(-\infty,1]$ (or equivalently, $H^{-1}(0,1]$) is a cylindrical domain for any $H \in U(M;V)$. The purpose of this section is to study the following functional on $U(M;V)$. 
\[
S:U(M;V) \to \R \qquad \text{given by}\qquad S(H) := \int_X \Delta H \cdot \mu_g \qquad\text{with}\qquad X := H^{-1}(-\infty,1]
\]

We begin by computing a useful formula for the variation of $S$. 

\begin{prop}[Variation] \label{prop:variation_of_S} The variation $\delta S$ of the functional $S:U(M;V) \to \R$ is given by
\[
\delta S_H(h) = \int_{\partial X} h \cdot (\frac{\on{dim}(M) + 2}{2} \cdot |\nabla H|^2 - 2\Delta H) \cdot \iota_V\mu_g
\]
\end{prop}

\begin{proof} Fix a function $H \in U(M;V)$ and a tangent vector $h \in \Gamma(M;V)$.  We set
\[H_t := H + t \cdot h \qquad\text{and}\qquad X_t := H_t^{-1}(-\infty,1]\] 
The variation $\delta S_H(h)$ of $S$ along $h$ is the time derivative of $S(H_t)$ at $t = 0$.
\[
\delta S_H(h) = \frac{d}{dt} \big(\int_{X_t} \Delta H_t \cdot \mu_g\big)|_{t=0} = \int_X \Delta \big(\frac{dH_t}{dt}\big)|_{t=0} \cdot \mu_g + \int_{\partial X} \Delta H \cdot \big(\iota_{\frac{dX}{dt}} \mu_g\big)|_{t=0} \]
Here $\frac{dX}{dt}$ is the variation of $X_t$ at $t = 0$, i.e. a vector field along $\partial X$ given as $\frac{d\Psi}{dt}$ for a family of parametrizations $\Psi_t:\partial X \simeq \partial X_t$. Note that this depends on $\Psi$, but $\iota_{\frac{dX}{dt}}\mu_g|_{\partial X}$ does not.

\begin{lemma} Under a specific parametrization of $\partial X_t$, the variation $\frac{dX}{dt}$ of $X_t$ at $t = 0$ is given by
\[\frac{dX}{dt} = -h \cdot V\]
\end{lemma}

\begin{proof} Recall that the flow $\Phi$ of $V$ determines a diffeomorphism $\Phi:\R_r \times \partial X \simeq M$. In these coordinates, $V = \partial_r$ and $H = e^{r}$. Furthermore, $H_t = f_t \cdot e^{r}$ where
\[f_t:\partial X \to \R \qquad\text{satisfies}\qquad f_0 = 1 \quad\text{and}\quad \frac{df}{dt} = h \cdot e^{-r} = h \cdot H^{-1}\]
For small $t$, the boundary $\partial X_t$ may be parametrized via
\[
\Psi_t:\partial X \to \R \times \partial X \qquad\text{with}\qquad \Psi_t(x) = (-\log(f_t(x)), x)
\]
The variation of the boundary $\frac{dX}{dt}$ under the parametrization $\Psi_t$ is thus
\[
\frac{dX}{dt} = -\frac{df}{dt}|_{t=0} \cdot f_0^{-1} \cdot \partial_r = -h \cdot V \qedhere
\] \end{proof}

Returning to the proof of Proposition \ref{prop:variation_of_S}, we apply our formulas for the variations of $X_t$ and $H_t$ to acquire the following expression.
\begin{equation} \label{eqn:variation_of_S_1}
\delta S_H(h) = \int_X \Delta h \cdot \mu_g - \int_{\partial X} h \cdot \Delta H \cdot \iota_V\mu_g
\end{equation}
We now proceed to analyze the first integral in (\ref{eqn:variation_of_S_1}). Using the divergence theorem and the fact that $|\nabla H| \cdot \nu = \nabla H$ on any regular level set of $H$, we may write
\begin{equation} \label{eqn:variation_of_S_2}
\int_X \Delta h \cdot \mu_g = \int_{\partial X} \langle \nu,\nabla h\rangle \cdot \iota_\nu\mu_g = \int_{\partial X}  \langle \frac{\nabla H}{|\nabla H|},\nabla h\rangle \cdot \iota_\nu\mu_g
\end{equation}
Next, we note that $|\nabla H| \cdot \langle V,\nu\rangle = 1$ and $\langle V,\nu\rangle \cdot \iota_\nu \mu_g = \iota_V\mu_g$ on $\partial X$. Therefore
\begin{equation} \label{eqn:variation_of_S_3}
\int_{\partial X}  \langle \frac{\nabla H}{|\nabla H|},\nabla h\rangle \cdot \iota_\nu\mu_g = \int_{\partial X}  \langle \nabla H,\nabla h\rangle \cdot \langle V,\nu\rangle \cdot \iota_\nu\mu_g = \int_{\partial X}  \langle \nabla H,\nabla h\rangle \cdot \iota_V\mu_g 
\end{equation}
Using the Leibniz rule for the covariant derivative $\nabla$, we thus find that
\begin{equation} \label{eqn:variation_of_S_4}
\int_X \Delta h \cdot \mu_g = \int_{\partial X}  \langle \nabla H,\nabla h\rangle \cdot \iota_V\mu_g = \int_{\partial X}  \on{div}(h \cdot \nabla H\rangle) \cdot \iota_V\mu_g - \int_{\partial X} h \cdot \Delta H \cdot \iota_V\mu_g 
\end{equation}

Now focus on the first integral on the righthand side. Since $\nabla H$ is linear of slope $0$ and $h$ is slope $1$, the divergence $\on{div}(h \cdot \nabla H)$ is linear of slope $1$. Therefore
\[\on{div}(h \cdot \nabla H) \cdot \mu_g \qquad \text{is a linear volume form of slope}\quad \frac{\on{dim}(M) + 2}{2}\]
Thus we apply Lemma \ref{lem:properties_of_linear_tensors}(c) to find that
\begin{equation} \label{eqn:variation_of_S_5}
\int_{\partial X}  \on{div}(h \cdot \nabla H) \cdot \iota_V\mu_g  = \frac{\on{dim}(M) + 2}{2} \cdot \int_X \on{div}(h \cdot \nabla H) \cdot \mu_g
\end{equation}
Finally, we once more apply Stokes' theorem to see that
\begin{equation} \label{eqn:variation_of_S_6}
\int_X \on{div}(h \cdot \nabla H) \cdot \mu_g = \int_{\partial X} \langle \nu, h \cdot \nabla H\rangle \cdot \iota_\nu\mu_g = \int_{\partial X} h \cdot |\nabla H|^2 \cdot \iota_V\mu_g
\end{equation}
Combining the formulas (\ref{eqn:variation_of_S_5}) and (\ref{eqn:variation_of_S_6}), and plugging the result into (\ref{eqn:variation_of_S_4}), we find that
\begin{equation} \label{eqn:variation_of_S_7}
\int_X \Delta h \cdot \mu_g = \frac{\on{dim}(M) + 2}{2} \cdot \int_{\partial X} h \cdot |\nabla H|^2 \cdot \iota_V\mu_g - \int_{\partial X} h \cdot \Delta H \cdot \iota_V\mu_g
\end{equation}
Plugging (\ref{eqn:variation_of_S_7}) into (\ref{eqn:variation_of_S_1}) concludes the proof.
\end{proof}

\noindent By applying the variational formula in Proposition \ref{prop:variation_of_S}, we can deduce a sandwiching property. 

\begin{prop}[Sandwich Estimate] \label{prop:sandwich_estimate} Let $G,H:M \to \R$ be maps in $U(M;V)$. Suppose that
\[\nabla V:TM \to TM \qquad\text{is self-adjoint}\]
\[\nabla^2 G \text{ and }\nabla^2 H \text{ are positive semi-definite} \qquad \text{and}\qquad G \le H \le L \cdot G \text{ for a constant }L \ge 1\]
Then $S(G)$ bounds $S(H)$ from above, up to a constant dependent on $L$ and the dimension $d$ of $M$.
\[S(H) \le C(L,d) \cdot S(G) \qquad\text{where}\qquad C(L,d) = \on{exp}(\frac{1}{2} \cdot Ld^2) \quad\text{and}\quad d = \on{dim}(M)\] 
\end{prop}

\begin{proof} Consider the family of functions and domains parametrized by $[0,1]_t$, given by
\[F_t = (1-t)\cdot H + t \cdot G = H + t \cdot (G - H) \qquad\text{and}\qquad X_t = F_t^{-1}(-\infty,1]\]
Due to our hypotheses on $G$ and $H$, $F_t$ and $X_t$ have the following properties.
\[\nabla^2 F_t \ge 0 \qquad F_s \ge F_t \quad\text{and}\quad X_s \subset X_t \qquad \text{for }s \le t\]
On $X_1$, we can bound the time derivative of $F$ from below as follows.
\[\frac{dF_t}{dt} = -|G - H| \ge -H \ge -\underset{X_1}{\on{max}}(H) \ge -L \cdot \underset{X_1}{\on{max}}(G) = -L\]
Moreover, by Lemma \ref{lem:V_linear_pos_def}, we know that
\[
2 \cdot \Delta F_t \ge |\nabla F_t|^2 \quad\text{on}\quad \partial X_t = F_t^{-1}(1) 
\]
Now we apply the formula for the variation of $S$ derived in Proposition \ref{prop:variation_of_S}.
\[
\frac{d}{dt}(S(F_t))|_{t=s} = \int_{\partial X_s} \frac{dF}{dt} \cdot (\frac{d + 2}{2} \cdot |\nabla F_s|^2 - 2 \cdot \Delta F_s) \cdot \iota_V\mu_g \ge -Ld \cdot \int_{\partial X_s} \Delta F_s \cdot \iota_V\mu_g 
\]
Now note that by Corollary \ref{cor:grad_div_laplacian_linear} $\Delta H$ is $V$-linear of slope $0$. Therefore, $\Delta H \cdot \mu_g$ is a volume form of slope $d/2$, and so by Lemma \ref{lem:V_linear_volume} we have
\[\int_{\partial X_s} \Delta F_s \cdot \iota_V\mu_g = \frac{d}{2} \cdot \int_{X_s} \Delta F_s \cdot \mu_g = \frac{d}{2} \cdot S(F_s)\]
Therefore, we acquire the following differential inequality for $S(F_s)$.
\[\frac{d}{dt}(S(F_t)) \ge -\frac{L d^2}{2} \cdot S(F_t)\]
Integrating this inequality from $0$ to $1$, we obtain the desired result.
\[C(L,d) \cdot S(G) = \exp(\frac{Ld^2}{2}) \cdot S(F_1) \ge S(F_0) = S(H) \qedhere\]\end{proof}   

\subsection{Standard Ellipsoids} \label{subsec:ellipsoids} The prototypical star-shaped, convex domains in $\C^n$ are standard ellipsoids. Here we review some facts about these domains that we will need for Theorem \ref{thm:main_inequality}.

\begin{definition} The \emph{standard ellipsoid} $E$ with symplectic widths $a_1 \le \dots \le a_n$ is the sub-level set
\[
E = H_E^{-1}(-\infty,1] \qquad \text{with} \qquad H_E:\C^n \to \R \qquad\text{given by}\qquad (z_1,\dots,z_n) = \pi \cdot \sum_i \frac{|z_i|^2}{a_i}
\]
\end{definition}

\noindent Every ellipsoid in $\C^n$ is symplectomorphic (via an affine symplectomorphism) to a standard one. Moreover, any convex body can be sandwiched between an ellipsoid and its scaling, as stated by John's ellipsoid theorem.

\begin{thm}[John Ellipsoid] \cite{j1948} \label{thm:john_ellipsoid} Let $K \subset \R^n$ be a convex domain. Then there exists a unique ellipsoid $E$ of maximal volume in $K$. Furthermore, if $c \in E$ is the center of $E$ then
\[E \subset K \subset c + n(E - c)\]
\end{thm}

\noindent In $\C^n$, we can assume that the John ellipsoid is standard after applying a symplectomorphism.

\begin{lemma} \label{lem:John_ellipsoid} \cite[Cor. 3.6]{ce2021} Let $X \subset \C^n$ be a convex domain. Then there is an affine symplectomorphism $\Phi:\C^n \to \C^n$ and a standard ellipsoid $E$ such that
\[
E \subset \Phi(X) \subset 2n \cdot E
\]
\end{lemma}

For ellipsoids, most of the geometric quantities that appear in the proof of Theorem \ref{thm:main_inequality} can be computed explicitly. We record the results of that computation.

\begin{lemma}[Ellipsoid Quantities] \label{lem:ellipsoid_quantities} Let $E$ be a standard ellipsoid with symplectic widths $a_1 \le \dots \le a_n$. Then the systole period, Laplacian integral and metric volume of $E$ are given by
\[c(E) = a_1 \qquad S(H_E) = \frac{4\pi}{n!} \cdot \big(\sum_i \frac{1}{a_i}\big) \cdot \prod_i a_i \qquad \on{vol}_g(E) = \frac{1}{n!} \cdot \prod_i a_i\]
In particular, these quantities obey the following inequalities.
\[
4\pi \cdot \on{vol}_g(E) \le c(E) \cdot S(H_E) \le 4\pi n \cdot \on{vol}_g(E)
\]
\end{lemma}

\begin{proof} The formula for $c(E)$ is standard (cf. \cite{gh2018}). To derive the volume, note that $E = A^{-1/2}(B^{2n})$ where $A$ is the diagonal Hermitian matrix with $H_E(z) = \langle z,Az\rangle$. Therefore
\[\on{vol}(E) = \on{det}(A)^{-1/2} \cdot \on{vol}(B^{2n}) = \Big(\prod_i \big(\frac{\pi}{a_i}\big)^2\Big)^{-1/2} \cdot \frac{\pi^n}{n!} = \frac{1}{n!} \cdot \prod_i a_i\]
To compute $S(H_E)$, we note that $\Delta H_E$ is constant and given by
\[
\Delta H_E = 4\pi \cdot \sum_i \frac{1}{a_i} \qquad\text{so that}\qquad S(H_E) = \Delta H_E \cdot \on{vol}_g(E)
\]
Finally, to prove the claimed inequality it suffices to note that $4\pi \le \Delta H_E \cdot c(E) \le 4\pi n$. Indeed, by the formulas already derived, we have
\[\Delta H_E \cdot c(E) = 4\pi \cdot \sum_i \frac{a_1}{a_i} \qquad\text{and}\qquad 1 \le \sum_i \frac{a_1}{a_i} \le n \qedhere\]
\end{proof}

\subsection{Proof Of Main Estimate} \label{subsec:proof_of_main_inequality} We are now ready to prove Theorem \ref{thm:main_inequality}. 

\begin{proof} Let $X \subset \C^n$ be a convex, star-shaped domain. By Lemmas \ref{lem:John_ellipsoid} and \ref{lem:ruelle_liouville_domain}, we may assume without loss of generality that there is a standard ellipsoid $E$ such that
\begin{equation}
\label{eq:proof_of_main_estimate_a}
\frac{1}{2n}\cdot E \subset X \subset E
\end{equation}
Note that since the systole period is a symplectic capacity on convex domains (cf. \cite{gh2018}), we have
\[c(X) \le c(E) \qquad \text{and}\qquad \on{vol}(E) \le (2n)^{2n} \cdot \on{vol}(X) \]
Now let $T\Phi:\R \times X \to X$ denote the symplectic cocycle induced by the Hamiltonian flow  $\Phi$ of $H_X$. This cocycle is generated by the Hessian, i.e.
\[
\frac{d}{dt}\big(T\Phi(t,x)) = \Omega \circ \nabla^2 H_X(\Phi(t,x)) \circ T\Phi(t,x)
\]
where $\Omega$ is the matrix representing multiplication by $i$. Convexity of $X$ implies that $\nabla^2 H_X$ is positive semi-definite. Thus we may apply the trace estimate, Lemma \ref{lem:trace_bound_Ru}, and conclude that
\[
\on{Ru}(X) \le \frac{8n^2}{\pi} \cdot \int_X \on{tr}\big(\nabla^2 H_X\big) \cdot \omega^n = \frac{8n^2}{\pi} \cdot n! \cdot \int_X \Delta H_X \cdot \mu_g = \frac{8n^2}{\pi} \cdot n! \cdot S(H_X)
\]
The inclusions \eqref{eq:proof_of_main_estimate_a} imply that
\[H_E\leq H_X \leq H_{\frac{1}{2n}\cdot E} = (2n)^2\cdot H_E\]

Now we apply the sandwiching estimate for $S$ derived in Proposition \ref{prop:sandwich_estimate}. Indeed, $M = \C^n \setminus 0$ has a cylindrical vector field $Z$ (the standard Liouville vector field) and the standard metric is $Z$-linear of slope $1$. Moreover, $Z$ satisfies
\[
\nabla Z = \frac{1}{2} \cdot \on{Id}
\]
We may therefore apply Proposition \ref{prop:sandwich_estimate} to find that
\[
S(H_X) \le C((2n)^2,2n) \cdot S(H_E) \qquad\text{where}\qquad C((2n)^2,2n) = \exp(8n^4)
\]
Finally, combing the estimates above and applying Lemma \ref{lem:ellipsoid_quantities}, we find that
\[
c(X) \cdot \on{Ru}(X) \le \frac{8n^2 \cdot n!}{\pi} \cdot \on{exp}(8n^4) \cdot c(E) \cdot S(H_E) \]
\[\le 32n^3 \cdot n! \cdot \on{exp}(8n^4) \cdot \on{vol}(E) \le 32n^3 \cdot (2n)^{2n} \cdot \on{exp}(8n^4) \cdot \on{vol}(X)
\]
This proves the inequality for the constant $C(n)$ given by
\[C(n) = 2^{2n+5} \cdot n^{2n+3} \cdot \on{exp}(8n^4) \qedhere\]
\end{proof}

\section{Ruelle Invariant Of Toric Domains} \label{sec:toric_formula} In this section, we compute the Ruelle invariant of toric domains in any dimension and explain the higher-dimensional examples of non-convex, dynamically convex domains. 

\subsection{Star-Shaped Toric Domains} \label{subsec:toric_domains} We begin by recalling the basics of toric domains.

\begin{remark} We recommend Gutt-Hutchings \cite{gh2018} for a detailed treatment. Also see \cite{ghr2020}. \end{remark} 

\noindent Consider $\C^n$ with the Hamiltonian action by $T^n \simeq (\R/\Z)^n$ induced by the $\on{GL}(n,\C)$-action.
\[
T^n \times \C^n \to \C^n \qquad \text{given by}\qquad \theta \cdot z = (e^{2\pi i\theta_1}z_1,\dots,e^{2\pi i \theta_n}z_n)
\]
This \emph{standard torus action} is generated by the following moment map. 
\[\mu:\C^n \to [0,\infty)^n \subset \R^n \qquad \text{given by}\qquad \mu = (\mu_1,\dots,\mu_n) \quad\text{and}\quad \mu_i(z) = \pi \cdot |z_i|^2\]
One can extend $\mu$ to a symplectomorphism on the free region of the action of the form
\[
(\mu,\theta):(\C^\times)^n \simeq (0,\infty)^n \times (\R/\Z)^n \simeq (0,\infty)^n \times T^n
\]
Here $(0,\infty)^n \times T^n$ has symplectic form $\sum_i d\mu_i \wedge d\theta_i$. 

\begin{definition} The \emph{toric domain} $X \subset \C^n$ with \emph{moment region} $\Omega \subset [0,\infty)^n$ is the $T^n$-invariant domain in $\C^n$ given by $X = \mu^{-1}(\Omega)$. It is conventional to use the notation $X_\Omega$ for $X$.
\end{definition} 

We are interested in toric domains that are also star-shaped. In the coordinates $(\mu,\theta)$, the Liouville vector field $Z$ and the Liouville form $\lambda$ on $\C^n$ are given by
\[Z = \sum_i  \mu_i\partial_{\mu_i} \qquad\text{and}\qquad \lambda = \sum_i \mu_i d\theta_i\]
Thus $X_\Omega$ is a star-shaped domain if and only if $\Omega$ is star-shaped with respect to $0$ and
\[\partial_+\Omega := \mu(\partial X_\Omega) \subset \partial\Omega \qquad\text{is transverse to }\sum_i \partial_i \mu_i \partial_{\mu_i}\]
The canonical Hamiltonian of a star-shaped toric domain $X$, its corresponding vector field and its Hamiltonian all possess nice toric formulas. We record these in the following lemma.

\begin{lemma} \label{lem:toric_formulas} Let $X$ be a star-shaped toric domain with moment region $\Omega$. Then
\begin{itemize}
    \item[(a)] The canonical Hamiltonian $H_X$ is given by $H_X = f_\Omega \circ \mu$ where
    \[f_\Omega:[0,\infty)^n \to \R \qquad\text{satisfies} \qquad \sum_i \mu_i \cdot \partial_if_\Omega = f_\Omega \qquad\text{and}\qquad f_\Omega^{-1}(1) = \partial_+\Omega\]
    \item[(b)] The Hamiltonian vector field $V_X$ of $H_X$ is given by
    \[V_X = \sum_i \partial_i f_\Omega \circ \mu \cdot \partial_{\theta_i}\]
    \item[(c)] The Hamiltonian flow of $H_X$ is given (in standard coordinates on $\C^n$) as
\[
\Phi(t,z) = U(t,\mu)z
\]
Here $U(t,\mu)$ is a diagonal, unitary matrix dependent only on $t$ and $\mu(z)$, with diagonal entries 
\[
u_j(t,\mu) = \exp\big(2\pi i t \cdot \partial_j f_\Omega(\mu)\big) \qquad\text{for}\qquad i = 1,\dots,n
\]
\item[(d)] The differential of the Hamiltonian flow of $H_X$ is given (in standard coordinates on $\C^n$) as
\[
T\Phi:\R \times \C^n \to \on{Sp}(2n) \qquad\text{where}\qquad T\Phi(t,z) = U(t,\mu)Q(t,z)
\]
Here $Q(t,z) = \on{Id} + t \cdot M(z)$ where $M(z)$ is a nilpotent matrix.
\end{itemize}
\end{lemma}

\begin{proof} To see (a), note that the toric formula for $Z$ implies that $F = f_\Omega \circ \mu$ satisfies $ZF = F$ and $F^{-1}(1) = \mu^{-1}(\partial_+\Omega) = \partial X$. Thus $H_X = F$ since these properties uniquely determine $H_X$. (b) and (c) are immediate from (a). To deduce (d), differentiate (c) to acquire the formula
\[
T\Phi(t,z) = U(t,\mu(z))(\on{Id} + t \cdot M(z)) \qquad\text{where}\qquad M(z)v = 2\pi i \cdot D(z) \circ \nabla^2f_\Omega(\mu(z)) \circ d\mu(v)
\]
Here $D(z)$ is the diagonal matrix with entries $z_i$. Note that $Q(t,z) = \on{Id} + t \cdot M(z)$ is symplectic for every $t$ and fixed $z$. By Lemma \ref{lem:symplectic_nilpotent}, $M(z)$ is nilpotent and $Q(t,z)$ has all $1$ eigenvalues.
\end{proof}

\begin{lemma} \label{lem:symplectic_nilpotent} Let $Q:\R \to \on{Sp}(2n)$ be a path of symplectic matrices of the form
\[Q(t) = \on{Id} + t M \qquad\text{where}\qquad M \text{ is $t$ independent}\]
Then $M$ is nilpotent and the eigenvalues of $Q(t)$ are $1$ for all $t \in \R$.
\end{lemma}

\begin{proof} Let $J$ be the matrix for the standard symplectic form on $\R^{2n}$. Then for all $t$, we have
\[(\on{Id} + tM)J(\on{Id} + tM)^T = J \qquad\text{or equivalently}\qquad MJ + JM^T = MJM^T = 0\]
Combining the last two formulas, we find that $M^2J = -MJM^T = 0$ so that $M^2 = 0$. Thus $M$ is nilpotent and has eigenvalues $0$. Hence $1$ is the only eigenvalue of $\on{Id} + tM$.
\end{proof}

We now calculate the Ruelle invariant of a star-shaped toric domain in any dimension. 

\begin{remark} Our formula specializes to the formulas in \cite{h2019,dgz2021} in dimension four. However, our calculation differs from both of \cite{h2019,dgz2021} and utilizes the properties of the Ruelle invariant in \S \ref{sec:Ruelle_density}. \end{remark}

\begin{prop}[Toric Ruelle] \label{prop:toric_ruelle} The Ruelle density $\on{ru}(X_\Omega)$ of a star-shaped toric domain $X_\Omega$ is given by
\[
\on{ru}(X_\Omega)(z) = \sum_i \partial_i f_\Omega(\mu(z)) 
\]
In particular, the Ruelle invariant of $X_\Omega$ is given by
\[
\on{Ru}(X_\Omega) = \sum_i \int_\Omega \partial_i f_\Omega \cdot \on{dvol}_{\R^n}
\]
\end{prop}

\begin{proof} We consider the lift to the universal cover of the differential $T\Phi$ of the Hamiltonian flow of the canonical Hamiltonian. By Lemma \ref{lem:toric_formulas} , we may write
\[
\widetilde{T\Phi}(t,z) = \widetilde{U}(t,\mu(z))\widetilde{Q}(t,z)
\]
By Proposition \ref{prop:Ruelle_invariant}(c) and the quasi-morphsim property for $r$, we can calculate the Ruelle density as the limit
\[
\on{ru}(X_\Omega)(z) = \lim_{T \to \infty} \frac{r(\widetilde{T\Phi}(T,z))}{T} = \lim_{T \to \infty} \frac{r(\widetilde{U}(T,\mu(z)))}{T} + \lim_{T \to \infty} \frac{r(\widetilde{Q}(T,z))}{T}
\]
Since $\widetilde{U}$ is already unitary and diagonal, so we see that
\[
\frac{r(\tilde{U}(t,z))}{T} = \frac{1}{2\pi i T} \cdot \int_0^T \on{tr}_\C(\frac{dU}{dt}(t,\mu(z)) \cdot \widetilde{U}^{-1}(t,\mu(z))) dt = \frac{1}{T} \int_0^T \sum_i \partial_i f_\Omega(\mu(z)) dt = \sum_i \partial_i f_\Omega(\mu(z))
\]
It now suffices to show that $r(\widetilde{Q}(T,z))/T \to 0$ as $T \to \infty$. Since the determinant quasimorphism $r$ and eigenvalue quasimorphism $e$ are equivalent (see \S \ref{subsec:rotation_quasimorphisms}), we have
\[
\lim_{T \to \infty} \frac{r(\widetilde{Q}(T,z))}{T} = \lim_{T \to \infty} \frac{e(\widetilde{Q}(T,z))}{T}
\]
Moreover, by Lemma \ref{lem:symplectic_nilpotent}, $Q(t)$ has all its eigenvalues equal to $1$ for all $t$. In particular, it has no eigenvalues on $(U(1) \cup (-\infty,0)) \setminus \{1\}$. Thus we conclude that 
\[\underline{e}(Q(t)) = 1 \in U(1) \qquad\text{and}\qquad e(\widetilde{Q}(t)) = 0 \qquad\text{for all }t \in \R \qedhere\]\end{proof}

\subsection{Monotone Toric Domains} \label{subsec:monotone_domains} In \cite{ghr2020}, Gutt-Hutchings-Ramos introduced the notion of a (strictly) monotone toric domain.

\begin{definition} \label{def:strictly_monotone} A star-shaped, toric domain $X_\Omega$ is \emph{strictly monotone} if either of the following equivalent conditions are satisfied.
\begin{itemize}
\item[(a)] the unit normal vector-field $\nu_\Omega:\partial_+\Omega \to \R^n$ pointing outward from $\Omega$ satisfies
\[\nu_\Omega(x) \in (0,\infty)^n \text{ for each }x \in \partial_+\Omega\]
\item[(b)] the gradient of the canonical function $f_\Omega:\Omega \to \R$ satisfies
\[\nabla f_\Omega(x) \in (0,\infty)^n \text{ for each }x \in \Omega \setminus 0\]
\end{itemize}
\end{definition}

In dimension four, a star-shaped toric domain is monotone if and only if it is dynamically convex \cite[Prop. 1.8]{ghr2020}. We generalize this result to higher dimensions, in one direction.

\begin{prop} \label{prop:strictly_monotone_implies_dynamically_convex} Let $X$ be a strictly monotone toric domain in $\C^n$. Then $X$ is dynamically convex. 
\end{prop} 

\begin{proof} Let $\gamma$ be a closed orbit of the Hamiltonian flow $\Phi$ of $H_X$ starting at $z \in \partial X$ with period $T$. We may assume (without loss of generality) that
\[
\C^n = \C^m \oplus \C^{n-m} \qquad\text{where}\qquad z_j = 0 \text{ if and only if }j = m+1,\dots,n
\]
By Corollary \ref{cor:dyn_cvx_Hamiltonian}, it suffices to show that the lower semicontinuous Conley-Zehnder index of $\gamma$ as a periodic orbit of $H_X$ is bounded below by $n$. That is, $\LCZ(X,H_X;\gamma) \ge n$. 

\vspace{3pt}

We start by analyzing the differential $T\Phi$ along $\gamma$. By Lemma \ref{lem:toric_formulas}(c) we may write
\[\Phi(t,z) = U(t,\mu(z))z \qquad \text{and}\qquad T\Phi(t,z) = U(t,\mu(z))Q(t,z)\]
Here $U(t,\mu(z))$ is a diagonal matrix with unit complex entries
\[u_j(t,\mu(z)) = \on{exp}(2\pi i t \cdot \partial_j f_\Omega(\mu(z)))\]
Note that the flow $\Phi$ preserves the symplectic subspace $\C^m \oplus 0$. Thus the differential preserves $\C^m \oplus 0$ and $(\C^m \oplus 0)^\omega = 0 \oplus \C^{n-m}$ and there is a block decomposition
\[
T\Phi = T\Phi_1 \oplus T\Phi_2 \qquad\text{with respect to the splitting }\C^m \oplus \C^{n-m}\text{ along }\gamma
\]
Since $U$ also decomposes in block form, it follows that we have a block decomposition
\[U = U_1 \oplus U_2 \qquad \text{and}\qquad Q = Q_1 \oplus Q_2 \]
A direct analysis of $Q$ shows that $Q_2(t,z) = \on{Id}$. Indeed, $Q(t,z) = \on{Id} + t \cdot M(z)$ where
\[M(z)v = 2\pi i \cdot D(z) \circ \nabla^2f_\Omega(\mu(z)) \circ d\mu(v)\]
Here $D(z)$ is the diagonal matrix with entries $z_j$. Since $z_j = 0$ for $j = m+1,\dots,n$, we can conclude that the lower block of $M(z)$ vanishes, and so $Q_2(t,z) = \on{Id}$. Finally, note that the period $T$ of $\gamma$ must satisfy 
\[
T \cdot \partial_j f_\Omega(\mu(z)) \in \Z_+ \qquad\text{for each}\qquad j = 1,\dots,m
\]
Thus, the upper block $U_1$ of $U$ satisfies $U_1(T,\mu(z)) = \on{Id}$, and is a closed loop in $\on{Sp}(2m)$. 

\vspace{3pt}

Now estimate the lower semi-continuous Conley-Zehnder index $\LCZ$ of the lift $\widetilde{T\Phi}:\R \times \C^n \to \widetilde{\on{Sp}}(2n)$ at $(T,z)$. By the above discussion, we may write 
\[
\LCZ(\widetilde{T\Phi}(T,z)) = \LCZ\left(\widetilde{U}_1(T,z) \widetilde{Q}_1(T,z) \oplus \widetilde{U}_2(T,z)\right) 
\]
By the additivity and Maslov index properties of $\LCZ$ (see Lemma \ref{lem:LCZ_index}(b)-(c)), we have
\begin{equation} \label{eqn:LCZ_of_TPhi}
\LCZ(\widetilde{T\Phi}(T,z)) = 2\mu(\widetilde{U}_1(T,z)) + \LCZ(\widetilde{Q}_1(T,z)) + \LCZ(\widetilde{U}_2(T,z))
\end{equation}
To bound the first and third term, note that $U$ is a diagonal unitary matrix, so we may write
\[\widetilde{U}_1(T,z) = \bigoplus_{j=1}^m \widetilde{u}_j(T,z) \qquad\text{and}\qquad \widetilde{U}_2(T,z) = \bigoplus_{j=m+1}^n \widetilde{u}_j(T,z)\]
By the direct sum property of the Maslov index \cite[Thm. 2.2.12]{ms2017} and the Conley-Zehnder index (see Lemma \ref{lem:LCZ_index}(b)), along with the calculation of $\LCZ$ in Lemma \ref{lem:LCZ_of_U1_element}, we may thus write
\begin{equation} \label{eqn:CZ_of_U1} \mu(\widetilde{U}_1(T,z)) = \sum_{j = 1}^m \mu(\widetilde{u}_j(T,z)) \ge m\end{equation}
\begin{equation} \label{eqn:CZ_of_U2} \LCZ(\widetilde{U}_2(T,z)) = \sum_{j=m+1}^n \LCZ(\widetilde{u}_j(T,z)) \ge n - m \end{equation}
For the second term, note that $\widetilde{Q}_1(T,z)$ satisfies $e(\widetilde{Q}_1(T,z)^k) = e(\widetilde{Q}_1(kT,z)) = 0$ where $e$ is the eigenvalue quasimorphism in Example \ref{ex:eigenvalue_quasimorphism}. Thus the homogenization $\rho$ of $e$, which is the unique homogeneous rotation quasimorphism, is also $0$ on $\widetilde{Q{}}_1(T,z)$. Then by Lemma \ref{lem:LCZ_index}(d)
\begin{equation} \label{eqn:CZ_of_Q1} \LCZ(\widetilde{Q}_1(T,z)) \ge \rho(\widetilde{Q}_1(T,z)) - m = -m\end{equation}
By plugging (\ref{eqn:CZ_of_U1}-\ref{eqn:CZ_of_Q1}) into (\ref{eqn:LCZ_of_TPhi}), we acquire the desired lower bound.
\[
\LCZ(X,H_X;\gamma) = \LCZ(\widetilde{T\Phi}(T,z)) \ge 2m - m + (n-m) = n \qedhere
\]
\end{proof}

\begin{remark} Although we will not require this property later in the paper, Proposition \ref{prop:toric_ruelle} implies that the Ruelle invariant of a strictly monotone domain is always positive.

\begin{cor} \label{cor:Ruelle_of_monotone} Let $X$ be a strictly monotone, star-shaped toric domain in $\C^n$. Then
\[\on{ru}(X) > 0 \qquad\text{and}\qquad \on{Ru}(X) > 0\]
\end{cor} \end{remark}

\subsection{Concave Toric Domains} \label{subsec:concave_domains} We are interested in the following sub-class of monotone domains.

\begin{definition} A star-shaped toric domain $X_\Omega$ is \emph{concave} if the complement of $\Omega$ is convex. 
\end{definition}

\begin{lemma} A smooth concave toric domain $X_\Omega$ is strictly monotone, and thus dynamically convex.
\end{lemma}

\begin{proof} It suffices to show that $\langle \nu_\Omega(x),e_i\rangle > 0$ for each unit basis vector $e_i$ and every $x \in \partial_+\Omega$. 

\vspace{3pt}

To prove this, let $K$ be the closure of $[0,\infty)^n \setminus \Omega$. Note that $\partial_+\Omega$ is a properly embedded smooth hypersurface in $[0,\infty)^n$ with $\partial_+\Omega \subset \partial K$. Moreover, the outward unit normal $\nu_\Omega$ to $\partial_+\Omega$ is normal and inward pointing along $\partial K$. Since $K$ is convex, this implies that 
\begin{equation} \label{eqn:concave_is_monotone_1}
\langle \nu_\Omega(x),w - x\rangle \ge 0 \qquad \text{for any $w \in K$ and $x \in \partial_+\Omega$}
\end{equation}
Since $\Omega$ is compact, $K$ contains the scaling $c \cdot e_i$ for every $i = 1,\dots,n$ and all $c > 0$ sufficiently large. Thus (\ref{eqn:concave_is_monotone_1}) implies that, for any $x \in \partial_+\Omega$, we have 
\[\langle \nu_\Omega(x),e_i\rangle \ge \frac{1}{c} \cdot \langle \nu_\Omega(x),x\rangle \qquad \text{for all $c > 0$ sufficiently large}\]
To finish the proof, note that $\langle x,\nu_\Omega(x)\rangle > 0$ for any $x \in \partial_+\Omega$ since
\[\langle \nu_\Omega(x),x\rangle = |\nabla f_\Omega(x)|^{-1} \cdot \langle x,\nabla f_\Omega(x)\rangle = |\nabla f_\Omega(x)|^{-1} \cdot f_\Omega(x) = |\nabla f_\Omega(x)|^{-1} > 0 \qedhere\]\end{proof}

We will need a formula for the minumum period of a Reeb orbit on the boundary of a concave toric domain given by Gutt-Hutchings \cite{gh2018}. Given a subset $S \subset \{1,\dots,n\}$, we adopt the notation
\[(0,\infty)^S := \{x \in [0,\infty)^n \; : \; x_i \neq 0 \text{ if and only if }i \in S\}\]
Given a star-shaped $\Omega \subset [0,\infty)^n$, we also let $\Omega_S \subset \Omega$ and $\partial_+\Omega_S \subset \partial_+\Omega$ be the subsets
\[\Omega_S := \Omega \cap (0,\infty)^S \qquad\text{and}\qquad \partial_+\Omega_S := \partial_+\Omega \cap \Omega_S\]
\begin{definition} The \emph{bracket} $[-]_\Omega$ of a concave, star-shaped moment region $\Omega$ is the function
\[
[-]_\Omega:[0,\infty)^n \to [0,\infty) \qquad\text{given by}\qquad [v]_\Omega := \on{min}\{\langle x,v\rangle \; : \; x \in \partial_+\Omega_{S}\} \qquad\text{if }v \in (0,\infty)^S 
\] \end{definition}

\begin{lemma} \label{lem:min_period_concave} \cite[\S 2.3, p. 22]{gh2018} Let $X_\Omega$ be a concave, star-shaped toric domain. Then
\[
c(X_\Omega) = \on{min}\{[v]_\Omega \; : \; v \in \Z_{\ge 0}^n \setminus 0\}
\]
\end{lemma}

\noindent Note that, if $\Omega$ and $\Xi$ are concave, star-shaped moment regions with $\Omega \subset \Xi$, then $[-]_\Omega \le [-]_\Xi$. Thus, as a corollary of Lemma \ref{lem:min_period_concave} we have 

\begin{cor} If $X$ and $X'$ are concave, star-shaped and toric and $X \subset X'$, then $c(X) \le c(X')$.
\end{cor}

\subsection{Counter-Examples} \label{subsec:counterexamples} We conclude this paper by constructing new non-convex, dynamically convex domains in $\C^n$ by generalizing the strain operation of Dardennes-Gutt-Zhang \cite{dgz2021}. 

\begin{prop} \label{prop:toric_counter_examples} Let $X_\Omega$ be a star-shaped, concave toric domain. Then for any $C,\epsilon > 0$, there is
\[\text{a smooth, star-shaped, concave moment region} \quad \hat{\Omega} \supset \Omega\]
that satisfies the following properties
\[\on{vol}(X_\Omega) \le \on{vol}(X_{\hat{\Omega}}) \le \on{vol}(X_\Omega) + \epsilon \qquad \on{Ru}(X_{\hat{\Omega}}) \ge C \qquad c(X_\Omega) \le c(X_{\hat{\Omega}})\]
\end{prop}

\begin{proof} We start by fixing some notation. Fix a large $B > 0$ such that the moment region
\[
\Xi := \big\{x \in [0,\infty)^n \; : \; B^{-1} \cdot x_2 + \dots + B^{-1} \cdot x_n \le 1\big\} \qquad\text{satisfies}\qquad \Omega \subset \Xi
\]
Also let $\Delta$ denote the moment region for a very flat ellipsoid, given by
\[\Delta := \big\{x \in [0,\infty)^n \; : \; A^n \cdot x_1 + A^{-1} \cdot x_2 + \dots + A^{-1} \cdot x_n \le 1\big\}\]
Here $A$ is a positive constant that we will specify below. The volume and Ruelle density of $X_\Delta$ can be calculated as
\[
\on{vol}(X_\Delta,\omega) = \on{vol}(\Delta) = \frac{1}{n!} \cdot A^{-1} \qquad\text{and}\qquad \on{ru}(\C^n,H_\Delta) =  A^n + (n-1) \cdot A^{-1} 
\]
Moreover, the volume of $\Delta \setminus \Xi$ can be estimated as
\[
\on{vol}(\Delta \setminus \Xi) = \on{vol}(\Delta) - \on{vol}(\Xi \cap \Delta) \ge \frac{1}{n!} \cdot A^{-1} - A^{-n} \cdot B^{n-1}
\]

Now let $\hat{\Omega}$ be a smooth, star-shaped, concave moment region given by a concave smoothing of the union $\Omega \cup \Delta$ such that
\[
\Omega \cup \Delta \subset \hat{\Omega} \subset (1+\frac{1}{A}) \cdot \Omega \cup \Delta \qquad\text{and}\qquad \hat{\Omega} \setminus \Xi = \Delta \setminus \Xi
\]
The only non-trivial bounds are the volume upper bound and the Ruelle invariant lower bound. For the volume bound, we note that for sufficiently large $A$, we have
\begin{equation}
\on{vol}(\hat{\Omega}) \le (1 + \frac{1}{A})^n \cdot \on{vol}(\Omega \cup \Delta) \le (1 + \frac{1}{A})^n \cdot (\on{vol}(\Omega) + \frac{A^{-1}}{n!}) \le \on{vol}(\Omega) + \epsilon
\end{equation}
For the Ruelle bound, we note that the Ruelle density of $X_{\hat{\Omega}}$, given by
\[
\on{ru}(\C^n,X_{\hat{\Omega}}) = \sum_i \partial_i f_{\hat{\Omega}} \ge 0
\] which is positive since $\hat{\Omega}$ is monotone. Moreover, we have
\[
f_{\hat{\Omega}} = f_{\Delta} \qquad\text{on}\qquad \hat{\Omega} \setminus \Xi = \Delta \setminus \Xi
\]
\[
\on{Ru}(X_{\hat{\Omega}}) = \sum_i \int_{\hat{\Omega}} \partial_i f_{\hat{\Omega}} \cdot \on{dvol}_{\R^n} \ge \sum_i \int_{\Delta \setminus \Xi} \partial_if_\Delta \cdot \on{dvol}_{\R^n} = (\frac{1}{n!} \cdot A^{-1} - A^{-n} \cdot B^{n-1}) \cdot (A^n + (n-1) \cdot A^{-1})
\]
For sufficiently large $A$, we can thus acquire $\on{Ru}(X_{\hat{\Omega}}) \ge C$. This concludes the proof. \end{proof}

\bibliographystyle{hplain}
\bibliography{dyn_cvx_2_bib}

\end{document}